\documentclass{birkjour_t2}
 \newtheorem{thm}{Theorem}[section]
 \newtheorem{cor}[thm]{Corollary}
 \newtheorem{lem}[thm]{Lemma}
 \newtheorem{prop}[thm]{Proposition}
 \theoremstyle{definition}
 
 \theoremstyle{remark}

 \numberwithin{equation}{section}

\usepackage[
  colorlinks=true,
  linkcolor={red!55!black},      %
  citecolor={orange!60!black},   %
  urlcolor={orange!70!black}     %
]{hyperref}
\usepackage{enumitem}
\usepackage[msc-links]{amsrefs}
\usepackage{mathtools}
\usepackage{amsmath,amssymb}
\usepackage{standalone}
\usepackage{tikz}
\usetikzlibrary{calc, arrows.meta, decorations.pathreplacing, patterns}
\usepackage{bm}
\usepackage[normalem]{ulem}
\usepackage{cancel}

\usepackage{xcolor}
\usepackage{tcolorbox}

\newcommand{\Wellposedness}{Well-posedness}
\newcommand{\cf}{cf.}

\newcommand{\C}{\mathbb{C}}
\newcommand{\R}{\mathbb{R}}
\newcommand{\N}{\mathbb{N}}

\newcommand{\statespace}{\bm{\mathcal{H}}}
\newcommand{\energyspace}{\bm{\mathcal{V}}}

\newcommand{\redstatespace}{\bm{\mathsf{H}}}
\newcommand{\redenergyspace}{\bm{\mathsf{V}}}

\DeclarePairedDelimiterX\setc[2]{\{}{\}}{\,#1 \;\delimsize\vert\; #2\,}
\DeclarePairedDelimiter\monosetc{\{}{\}}

\newcommand{\jmp}[1]{[\![#1]\!]}

\newcommand{\vve}{{\bm{e}}}
\newcommand{\vvn}{{\bm{n}}}
\newcommand{\vvnu}{{\bm{\nu}}}
\newcommand{\vvH}{{\bm{H}}}
\newcommand{\vvE}{{\bm{E}}}
\newcommand{\vvU}{{\bm{U}}}
\newcommand{\vvu}{{\bm{u}}}
\newcommand{\vvV}{{\bm{V}}}
\newcommand{\vvW}{{\bm{W}}}

\newcommand{\vvv}{{\bm{v}}}
\newcommand{\vvJ}{{\bm{J}}}
\newcommand{\vvF}{{\bm{F}}}
\newcommand{\vvG}{{\bm{G}}}
\newcommand{\vvR}{{\bm{R}}}

\newcommand{\vvL}{{\bm{L}}}
\newcommand{\vvPhi}{{\bm{\Phi}}}
\newcommand{\vvPsi}{{\bm{\Psi}}}
\newcommand{\vvTheta}{{\bm{\Theta}}}
\newcommand{\vvP}{{\bm{P}}}
\newcommand{\vvLambda}{\bm{\Lambda}}

\DeclareMathOperator{\curl}{\mathbf{curl}}
\DeclareMathOperator{\scurl}{curl}
\renewcommand{\div}{\operatorname{div}}
\DeclareMathOperator{\grad}{\mathbf{grad}}
\DeclareMathOperator{\pcurl}{\widetilde{\curl}}
\DeclareMathOperator{\pdiv}{\widetilde{\div}}

\newcommand{\imp}{{\operatorname{imp}}}

\newcommand{\kernel}{\operatorname{Ker}}
\newcommand{\range}{\operatorname{Range}}

\NewDocumentCommand{\restrict}{s O{} m m}{
    \IfBooleanTF{#1}{
        \left.{#3}\right|_{#4}
    }{
        {#3}#2|_{#4}
    }
}

\newcounter{saveeq}      %
\newcounter{savesubeq}   %
\newcounter{savecureq}   %

\newcommand{\subeqpause}{%
  \setcounter{saveeq}{\value{parentequation}}%
  \setcounter{savesubeq}{\value{equation}}%
}

\newenvironment{resumesubeqs}{%
  \setcounter{savecureq}{\value{equation}}%
  \setcounter{equation}{\value{saveeq}}%
  \addtocounter{equation}{-1}%
  \begin{subequations}%
  \setcounter{equation}{\value{savesubeq}}%
}{%
  \end{subequations}%
  \setcounter{equation}{\value{savecureq}}%
}

\begin{document}

\title[Maxwell's equations with interface damping]
 {\Wellposedness{}, regularity and stability for linear Maxwell's equations with interface damping}

\author[Dörner]{Julian Dörner}
\address{%
  Department of Mathematics\\
  Karlsruhe Institute of Technology\\
  P.O. Box 6980\\
  76049 Karlsruhe,  Germany 
}
\email{julian.doerner@kit.edu}

\author[Nicaise]{Serge Nicaise}
\address{%
  Univ. Polytechnique Hauts-de-France\\
  INSA Hauts-de-France\\
  CERAMATHS - Laboratoire de Matériaux Céramiques et de Mathématiques\\
 F-59313 Valenciennes, France}
\email{Serge.Nicaise@uphf.fr}

\author[Schnaubelt]{Roland Schnaubelt}
\address{%
  Department of Mathematics\\
  Karlsruhe Institute of Technology\\
  P.O. Box 6980\\
  76049 Karlsruhe,  Germany 
}
\email{schnaubelt@kit.edu}



\date{\today}

\begin{abstract}
We study linear Maxwell's equations on a cuboid that is split into two parts by a planar interface, across which the tangential magnetic field jumps proportionally to the tangential electric field.
Conditions of this type arise in the modeling of atomically thin metamaterials, where the sheet is replaced by an effective surface current.
Working in a state space that admits a surface charge on the interface, we show that the problem is governed by a contraction semigroup.
To this end, we provide density results, Weber-type inequalities, and trace estimates for the underlying function spaces.
We further prove a hidden regularity result, stating that the fields in the domain of the Maxwell operator are piecewise Sobolev regular of first order.
Since the damping acts on the interface only, the energy does not decay for all initial fields.
We decompose the dynamics into a unitary part, carried by the eigenmodes that are invisible to the damping, and a strongly stable part.
In the non-refracting case with a non-resonant interface position, we characterize the decaying initial fields by two moment conditions.
\end{abstract}

\maketitle

\section{Introduction}
The Maxwell system is the foundation of electromagnetic theory. When describing composed materials, it becomes an interface
problem, which poses deep analytical challenges. In this paper we consider linear and isotropic Maxwell's equations on a
cuboidal domain with an impedance transmission condition, and we investigate regularity and stability of solutions.
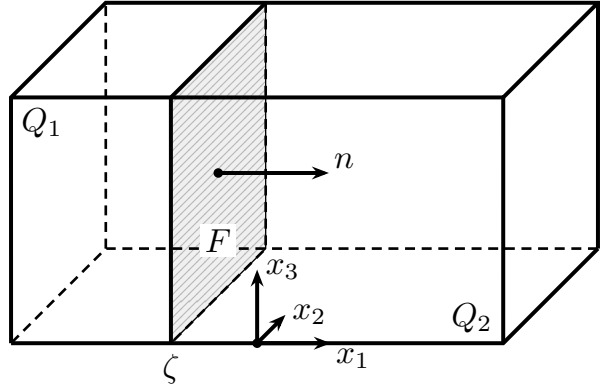
\begin{figure}
    \centering
    \scalebox{1.25}{
        \begin{tikzpicture}[
        scale=1.30,   %
        >={Stealth[length=1.8mm]},
        bnd/.style={very thick, black},
        hid/.style={thick, black, densely dashed},
        axis/.style={thin, black},
    ]

    \coordinate (b1) at ( 2,0,0);
    \coordinate (b2) at (-2,0,0);
    \coordinate (b3) at (-2,0,2);
    \coordinate (b4) at ( 2,0,2);
    \coordinate (t1) at ($(b1) + (0,2,0)$);
    \coordinate (t2) at ($(b2) + (0,2,0)$);
    \coordinate (t3) at ($(b3) + (0,2,0)$);
    \coordinate (t4) at ($(b4) + (0,2,0)$);

    \coordinate (shift) at (-0.35,0,0);
    \coordinate (i1) at ($(b2) + (1.65,0,0) + (shift)$);
    \coordinate (i2) at ($(b3) + (1.65,0,0) + (shift)$);
    \coordinate (i3) at ($(i2) + (0,2,0)$);
    \coordinate (i4) at ($(i1) + (0,2,0)$);

    \fill[gray!12] (i1) -- (i2) -- (i3) -- (i4) -- cycle;
    \draw[pattern=north east lines, pattern color=gray!55]
        (i1) -- (i2) -- (i3) -- (i4) -- cycle;

    \draw[hid] (b1) -- (b2) -- (b3);
    \draw[hid] (b2) -- (t2);
    \draw[hid] (i4) -- (i1) -- (i2);
    \draw[bnd] (i2) -- (i3) -- (i4);
    \draw[bnd] (b3) -- (b4) -- (b1);
    \draw[bnd] (t1) -- (t2) -- (t3) -- (t4) -- cycle;
    \draw[bnd] (b1) -- (t1);
    \draw[bnd] (b3) -- (t3);
    \draw[bnd] (b4) -- (t4);

    \coordinate (im) at ($0.5*(i2) + 0.5*(i4)$);
    \fill (im) circle (1.1pt);
    \draw[bnd,->] (im) -- ($(im) + (0.9,0,0)$)
        node[anchor=south west, inner sep=1pt] {$n$};

    \coordinate (o) at (0,0,2);
    \fill (o) circle (1.1pt);
    \draw[axis,->,very thick] (o) -- ($(o) + (0.6,0,0)$) node[anchor=north west, inner sep=1pt] {$x_1$};
    \draw[axis,->,very thick] (o) -- ($(o) + (0,0,-0.6)$) node[anchor=west, inner sep=1.5pt] {$x_2$};
    \draw[axis,->,very thick] (o) -- ($(o) + (0,0.6,0)$) node[anchor=west, inner sep=2pt] {$x_3$};

    \node[anchor=north, inner sep=3pt] at (i2) {$\zeta$};
    \node[anchor=north west, inner sep=3pt] at (t3) {$Q_1$};
    \node[anchor=south east, inner sep=3pt] at (b4) {$Q_2$};
    \node[fill=white, inner sep=1.5pt] at ($(im) + (0,-0.55,0)$) {$F$};

\end{tikzpicture}
    }
    \caption{Domain \(Q\) with sub-cubes \(Q_1\), \(Q_2\), and interface \(F\).}
    \label{fig:domain}
\end{figure}

Let \(Q \subset \R^3\) denote our domain of interest, which is partitioned by two sub-cubes \(Q_1\) and \(Q_2\) defined by
\begin{equation*}\label{eq:domain}
    Q = (-1,1)\times (0,1)^2,
    \quad
    Q_1 = (-1,\zeta)\times (0,1)^2,
    \quad
    Q_2 = (\zeta,1) \times (0,1)^2,
\end{equation*}
where \(\zeta \in (-1,1)\) is the point where the interface \(F = \monosetc{\zeta} \times (0,1)^2\) intersects the \(x_1\)-axis perpendicularly.
The unit normal vector on \(F\) is denoted by \(\vvn = (1,0,0)\).
If convenient, we identify \(F\) with the open unit square in two dimensions and use the same symbol.
See Figure~\ref{fig:domain} for the geometric setup.
For \(i\in \monosetc{1,2}\), we consider Maxwell's equations in the form
\begin{subequations}\label{eq:PDE}
\begin{equation}\label{eq:CurlEquations}
    \epsilon^{(i)}\partial_t \vvE^{(i)}(t) = \curl \vvH^{(i)}(t),
    \quad
    \mu^{(i)}\partial_t\vvH^{(i)}(t) = - \curl \vvE^{(i)}(t),
    \qquad \text{ in } (0,\infty)\times Q_i.
\end{equation}
Here, the upper index indicates the restriction \(f^{(i)} = \restrict{f}{Q_i}\) for any function \(f\,:\, Q\to \R\).
The scalar permittivity \(\epsilon\) and permeability \(\mu\) are piecewise constant and positive, that is,
\begin{equation}
    \epsilon^{(i)}, \mu^{(i)} > 0
    \qquad \text{ for } i \in\monosetc{1,2}.
\end{equation}
The equations are supplemented with the following transmission condition at the interface \(F\)
\begin{equation}\label{eq:TangentialJumps}
    \jmp{\vvE_\tau(t)} = 0,
    \quad
    \jmp{\vvH_\tau(t)} = -\sigma \vvE_\parallel(t)
    \qquad \text{on } (0,\infty)\times F
\end{equation}
with positive conductivity \(\sigma > 0\).
\subeqpause
\end{subequations}
Here, we denote by double brackets the jump \(\jmp{f} = \restrict{f^{(2)}}{F} - \restrict{f^{(1)}}{F}\) for any function \(f\,:\, Q\to \R\) that admits a well-defined trace on \(F\).
Furthermore, for any vector field \(\vvV\,:\, Q\to \R^3\), we use the shorthand notation for the tangential component and the rotated tangential component
\begin{equation}\label{eq:NotationalSugarTang}
    \restrict{\vvV^{(i)}_\parallel}{F} 
        = \restrict{\vvn\times \bigl(\vvV^{(i)}\times\vvn\bigr)}{F},
    \quad
    \restrict{\vvV^{(i)}_\tau}{F} = \restrict{\vvV^{(i)}\times \vvn}{F},
    \qquad
    \text{on } F,\quad
    \text{for } i\in\monosetc{1,2}.
\end{equation}
Note that the second condition in \eqref{eq:TangentialJumps} is unambiguous since the tangential component of the electric field is continuous by the first condition in \eqref{eq:TangentialJumps}. We therefore drop the superscript here and do not specify from which side the trace is taken.
Finally, the equations are supplemented with electric boundary conditions
\begin{resumesubeqs}
\begin{equation}\label{eq:ElectricBoundaryConditions}
    \vvE(t) \times \vvnu = 0,
    \qquad \text{in } (0,\infty) \times \partial Q,
\end{equation}
\end{resumesubeqs}
where \(\vvnu\) denotes the outer unit normal, 
and the initial conditions \(\vvE(0) = \vvE_0\) and \(\vvH(0) = \vvH_0\).

\subsection*{Physical relevance}
The model introduced in \eqref{eq:PDE} is used to study the optical properties of atomically thin metamaterials. 
Metamaterials are artificially engineered materials with properties beyond those found in naturally occurring media, and they have attracted broad interest because of their potential applications. In particular, atomically thin sheets of such materials can be designed to manipulate electromagnetic waves in ways that enable novel optical devices, see \cite{Cho.L.S.Et.2010,Ted.L.O.2016}. To study and model these systems, one often uses Maxwell's equations supplemented with effective dissipation models that account for the negligible thickness of the material compared to the relevant spatial scales.
These systems are modelled using the usual transmission conditions, see, e.g., \cite{Jac.1975.CE}, at a dielectric interface \(F\), which are given by
\begin{equation*}
    \jmp{\vvH_\tau(t)} = \vvJ_F(\vvE_\parallel)(t),
    \quad
    \jmp{\vvE_\tau(t)} = 0,
    \qquad \text{on } (0,\infty)\times F.
\end{equation*}
The surface current \(\vvJ_F\) describes the effective dissipation of the metamaterial.
Many dissipation models for \(\vvJ_F\) are considered in the literature, such as the Drude or Lorentz model, see, for example, \cite{Wer.W.M.Et.2015,Dep.2016.GOES}.
In this work, we consider the linear relation with constant conductivity, that is \(\vvJ_F(\vvE_\parallel)(t) = -\sigma \vvE_\parallel(t)\), which can be interpreted as an instantaneous Drude model and serves as a natural starting point to consider more complex models in the future.

\subsection*{Literature review}
Maxwell's equations at a material interface have attracted considerable interest in recent years.
A comprehensive local wellposedness theory was developed in \cite{Sch.Spi.2022}
for quasi-linear instantaneous material laws and a prescribed surface current at
a smooth interface, whereas \cite{Dor.Z.2023} studied the linear problem with a prescribed
surface current on an interface that touches the boundary. Existence, uniqueness and stability of semi-linear retarded Maxwell
systems with homogeneous transmission conditions were established in \cite{Doh.Ion.Wau.2024} within the framework of the theory of
evolutionary equations. In contrast to these works, the surface current in \eqref{eq:TangentialJumps} is not given in advance,
but is coupled to the tangential electric field, so that the interface itself acts as the only damping mechanism of the system.

In this respect, our transmission condition is the interface analogue of an impedance, or Silver--Müller, boundary condition, which has
been studied thoroughly in the past. For linear Maxwell equations with linear and non-linear variants of such boundary conditions
we mention \cite{Ell.L.N.2002,Ell.Lag.Nic.2002a,Cag.Ell.2011}, and the combination with interior conductivity models is investigated in
\cite{Nic.2018,Nic.Sch.2023}. There are also extensions to retarded boundary laws, cf.\ \cite{Ani.Pok.2019}, as well as to non-linear
Maxwell equations with non-linear boundary damping, see \cite{Sch.Spi.2021, Nut.Sch.2024}.
Since the damping in \eqref{eq:TangentialJumps} acts on the interface only, and hence on a set that does not surround the domain, the decay behaviour differs
substantially from the boundary damped case. The technical basis for our decay proof is the theory of elastic systems stabilized by
collocated feedback developed in \cite{Amm.Nic.2015}.

The regularity analysis of Maxwell equations has been intensively studied as well, see \cite{Cos.Dau.Nic.1999,Cos.Dau.2000,Loh.Nic.2002,Mel.Wor.2025}.
These results are not directly applicable in our situation, since the fields dictated by \eqref{eq:TangentialJumps} are divergence-free
and curl-regular in each sub-cube only, and the interface touches the outer boundary equipped with the boundary conditions
of a perfect conductor. To deal with these difficulties, we also use techniques developed for the analysis of
alternating direction splitting
 methods, see \cite{Zer.Jah.2023,Zer.2022b,Zer.2021}, which are well suited to our specific geometry and are cited below frequently.

\subsection*{Contributions, challenges and methods}
We develop a toolbox for the function spaces dictated by the transmission condition, whose elements are divergence-free
or possess curl-regularity in each  sub-cube only. Hence, a surface charge may build up at the interface. Moreover, a priori
the traces on $F$ just exist in suitable variants of $H^{-1/2}$ and their traces on $\partial F$ can be captured only in
an indirect way. Still we manage to work in a setting, where the fields possess square-integrable tangential traces on $F$.
Among our results are density properties, Weber-type inequalities, and trace estimates, by which the surface charge and the
tangential electric field control each other on the interface. Our arguments rely on subtle trace theorems in critical regularity,
which we combine with the theory of interpolation scales. 
We use various Helmholtz decompositions and elaborate elliptic regularity results for interface problems from \cite{Zer.2021}.

On this basis, we prove well-posedness of the problem, in the sense that it is governed by a contraction semigroup, and show that the
surface charge is transported by the induced surface current. We then establish a hidden regularity result, which states that
the fields are piecewise Sobolev regular of first order, uniformly with respect to the graph norm of the Maxwell operator.
To derive this sharp and somewhat unexpected result, we adapt techniques known for impedance boundary conditions \cite{Buf.C.2001a, Buf.C.2001b, Buf.C.S.2002}, where the interface setting also demands a careful analysis of the compatibility conditions at the edges of the interface.

Finally, we prove a stability result. Since the damping acts on the interface only, energy decay cannot be expected for all initial
fields. So we show that every trajectory splits orthogonally into a unitary part, carried by the eigenmodes that are invisible
 to the damping, and a strongly stable part.  Our spectral analysis is based on the compactness
    of the resolvent of the generator which is a consequence of the hidden regularity.
For non-refracting material parameters and a non-resonant interface position, we compute the eigenmodes explicitly and characterize the stable initial fields by two moment conditions along the axis of the cuboid. 
To this aim, we determine an orthonormal basis  of eigenfunctions of a curl-curl-type operator with piecewise constant coefficients, which yield the TE- and TM-modes of our Maxwell problem, cf.\ \cite{Cos.Dau.2019} for the case $\varepsilon=1=\mu$.

\subsection*{Outline}
The paper is organized as follows.
The remainder of this introduction collects the function spaces and the notation used throughout.
In Section~\ref{sec:WellPosedness}, we introduce the state and energy spaces, discuss the interface conditions, and prove well-posedness of \eqref{eq:PDE} as well as the surface continuity equation.
Section~\ref{sec:FunctionSpacesStudy} is devoted to the properties of the underlying function spaces, that is, Hodge projections, trace estimates, Weber inequalities, and density results, which are used throughout the manuscript.
The hidden regularity of the domain of the Maxwell operator is established in Section~\ref{sec:Regularity}.
Building on it, Section~\ref{sec:Stability} contains the spectral analysis of the Maxwell operator and the decomposition of the semigroup into a unitary and a strongly stable part.
Section~\ref{sec:CavityEigenmodes} characterizes the cavity eigenmodes in the non-refracting case and identifies the unobserved modes, completing the proof of Theorem~\ref{thm:FirstOrderStability}.

\subsection*{Basic function spaces \& Notation}

To close the introduction, we recall some essential Hilbert spaces.
Let \(D\subset \R^n\) be a Lipschitz domain.
For any \(k\in \N\) we denote with \(H^k(D)\) the Sobolev space of functions with weak derivative up to order \(k\) in \(L^2(D)\).
The closure of compactly supported test functions \(C^\infty_c(D)\) in \(H^1(D)\) is denoted with \(H^1_0(D)\), and we define for \(s\geq 0\) the fractional Sobolev spaces as
\begin{equation*}
    H^s(D) = \bigl[L^2(D), H^1(D)\bigr]_{s,2},
\end{equation*}
where \(\bigl[\cdot,\cdot\bigr]_{\theta,2}\) denotes the real interpolation functor.
See, for example, \cite{Kre.P.S.1982.ILO,Tri.1978.ITFS}.
Additionally, we make use of the half-integer spaces
\begin{equation}\label{eq:InterpolationSpaces}
    H^{1/2}_{00}(D) = \bigl[L^2(D), H^1_0(D)\bigr]_{1/2,2},
    \qquad
    H^{3/2}_0(D) = \bigl[L^2(D), H^2(D)\cap H^1_0(D)\bigr]_{3/4,2}
\end{equation}
to account for the impact of boundary conditions at the critical regularity levels \(1/2\) and \(3/2\).
We denote the topological dual of \(H^{1/2}_{00}(F)\) with pivot space \(L^2(F)\) as \(H^{-1/2}_{00}(F)\).
The topological dual of \(H^{1/2}(F)\) with the same pivot space is denoted as \(H^{-1/2}(F)\). 
Propositions 2.1.44, 2.1.47 and 2.1.59 of \cite{Ass.C.L.2018.MFCE} show that $H^{1/2}_{00}(F)$ is the space of functions
$u\in H^{1/2}(F)$ whose 0-extension $\tilde u$ belongs to  $H^{1/2}(\partial Q_i)$ for $i=1$ or $i=2$. Moreover,
$H^{1/2}_{00}(F)$  is densely embedded into  $H^{1/2}(F)=H^{1/2}_0(F)$, but is strictly smaller, so that
$H^{-1/2}(F) \hookrightarrow H^{-1/2}_{00}(F)$. We define the restriction of $\varphi\in H^{-1/2}(\partial Q_i)$ to, say, $F$
by $\restrict{\varphi}{F}(u)=\langle \tilde u, \varphi\rangle_{H^{1/2}(\partial Q_i)}$
for $u \in H^{1/2}_{00}(F)$, see (2.9) in  \cite{Ass.C.L.2018.MFCE}.

We use the following Hilbert spaces to deal with vector fields and refer to \cite{Ass.C.L.2018.MFCE} for a detailed introduction:
\begin{align*}
    \vvH(\curl, Q) 
        &= \setc{\vvV\in L^2(Q)^3}{\curl \vvV\in L^2(Q)^3},\\
    \vvH_0(\curl,Q)
        &= \setc{\vvV\in \vvH(\curl,Q)}{\vvV\times \vvnu = 0 \text{ on } \partial Q},\\
    \vvH(\div_\theta, Q)
        &= \setc{\vvV\in L^2(Q)^3}{\div (\theta \vvV) \in L^2(Q)},\\
    \vvH_0(\div_\theta, Q)
        &= \setc{\vvV\in \vvH(\div_\theta, Q)}{ (\theta\vvV)\cdot \vvnu = 0 \text{ on } \partial Q}.
\end{align*}
These spaces are endowed with the corresponding graph norms.
Throughout, \(\theta\) is a piecewise constant, positive material parameter and the traces are understood in the sense of \(H^{-1/2}\).
We drop the lower index if \(\theta \equiv 1\).

Due to the inhomogeneous transmission conditions, we deal with vector fields that are discontinuous at the interface \(F\).
For that, we introduce the notion of functions that exhibit a piecewise curl and divergence
\begin{align*}
    \vvH(\pcurl,Q) 
        &= \setc{\vvV\in L^2(Q)^3}{\curl \vvV^{(i)} \in L^2(Q_i)^3,\,i\in\monosetc{1,2}},\\
    \vvH(\pdiv_\theta, Q)
        &= \setc{\vvV\in L^2(Q)^3}{\div (\theta^{(i)}\vvV^{(i)} )\in L^2(Q_i),\,i\in\monosetc{1,2}}.
\end{align*}
Note, that the definition of the latter space is independent of \(\theta\), i.e. \(\vvH(\pdiv_\theta,Q) = \vvH(\pdiv,Q)\).
We keep the notation in analogy to \(\vvH(\div,Q)\).
We emphasize the following characterization
\begin{subequations}\label{eq:ConnectionToFullSpaces}
    \begin{align}
        \vvH(\curl, Q) 
            &= \setc{\vvV\in \vvH(\pcurl,Q)}{\jmp{\vvV_\tau} = 0,\text{ in } H^{-1/2}_{00}(F)^2},\label{eq:ConnectionToFullSpaces:Curl}\\
        \vvH(\div_\theta, Q)
            &= \setc{\vvV\in \vvH(\pdiv_\theta,Q)}{\jmp{\theta\vvV_\perp} = 0,\text{ in } H^{-1/2}_{00}(F)}\label{eq:ConnectionToFullSpaces:Div}
    \end{align}
\end{subequations}
and refer to \cite[Prop.~2.2.30-2.2.32]{Ass.C.L.2018.MFCE} for the details.
Here, the tangential jump maps \(\vvH(\pcurl,Q)\) continuously into \(H^{-1/2}_{00}(F)^2\) by equation (2.24) of \cite{Ass.C.L.2018.MFCE}.
However, \(\vvV \in \vvH_0(\curl,Q)\) has rotated tangential trace \(\vvV_\parallel\) in \(H^{-1/2}(F)\) in view of Proposition~2.1.60 in \cite{Ass.C.L.2018.MFCE}, see page 99 therein.

In analogy with \eqref{eq:NotationalSugarTang}, we define the shorthand notation of the normal component of \(\vvV\,:\, Q\to \R^3\) as
\begin{equation}\label{eq:NotationalSugarNorm}
    \restrict{\theta^{(i)}\vvV^{(i)}_\perp}{F}  
        = \restrict{\theta^{(i)}\vvV^{(i)} \cdot \vvn}{F},
    \qquad
    \text{on } F,\quad
    \text{for } i\in\monosetc{1,2}.
\end{equation}
Finally, for \(k\in \N\), we introduce the space of functions that are piecewise Sobolev regular
\begin{equation*}
    PH^k(Q) = \setc{v\in L^2(Q)}{v^{(i)}\in H^k(Q_i)}
\end{equation*}
equipped with the canonical norm.

Throughout the manuscript we denote vector-valued quantities in bold, and write \(\alpha \lesssim \beta\) if there exists a constant \(C>0\), which is independent of \(\alpha\), \(\beta\), and time \(t\), such that \(\alpha \leq C\beta\).

\section{\Wellposedness{}}\label{sec:WellPosedness}

\subsection*{State space}

Along with the curl equations \eqref{eq:CurlEquations} comes Gauss's law, which we summarize in the following lemma.
The first two parts follow from \eqref{eq:CurlEquations} by applying the divergence, and the third part is a consequence of \eqref{eq:CurlEquations} and \cite[Thm.~3.5.1]{Ass.C.L.2018.MFCE}.
\begin{lem}\label{lem:GaussLaw}
    Let \(\bigl(\vvE,\vvH\bigr) \in C^1\bigl([0,\infty), L^2(Q)^6\bigr) \cap C^0\bigl([0,\infty), \vvH_0(\curl,Q)\times \vvH(\pcurl, Q)\bigr)\) solve \eqref{eq:PDE} with initial conditions
    \(\vvE(0)=\vvE_0, \vvH(0)=\vvH_0\).
    \begin{enumerate}[label=\normalfont (\alph*)]
        \item\label{lem:GaussLaw:E} If \(\pdiv\bigl(\varepsilon \vvE_0\bigr) = 0\), 
            then \(\pdiv \bigl(\varepsilon \vvE(t)\bigr) = 0\),  for all \(t> 0\).
        \item\label{lem:GaussLaw:H} If \(\div\bigl(\mu\vvH_0\bigr) = 0\), 
            then \(\div\bigl(\mu\vvH(t)\bigr) = 0\), for all \(t>0\).
        \item\label{lem:GaussLaw:Ntrace} If \(\vvH_0\cdot \vvnu = 0\), 
            then \(\vvH(t)\cdot \vvnu = 0\), for all \(t>0\).
    \end{enumerate}
\end{lem}

Motivated by Lemma~\ref{lem:GaussLaw}, we set
\begin{align*}
    \vvH_0(\div_\mu = 0, Q) &= \setc{\vvV\in \vvH_0(\div_\mu,Q)}{\div(\mu\vvV) = 0},\\
    \vvH(\pdiv_\varepsilon = 0, Q) &= \setc{\vvV\in \vvH(\pdiv_\varepsilon, Q)}{\pdiv(\varepsilon\vvV) = 0},
\end{align*}
and define the state space
\begin{equation*}
    \statespace = \vvH(\pdiv_\epsilon = 0, Q) \times \vvH_0(\div_\mu = 0, Q).
\end{equation*}
It is a closed subspace of \(L^2(Q)^6\) and, thus, a Hilbert space with the scaled inner product
\begin{equation}
    \Bigl(\bigl(\vvE,\vvH\bigr),\bigl(\vvU,\vvV\bigr)\Bigl)_{\statespace}
        = \int_{Q} \epsilon \vvE\cdot \overline{\vvU} +
            \mu \vvH \cdot \overline{\vvV},
        \quad \bigl(\vvE,\vvH\bigr),\bigl(\vvU,\vvW\bigr)\in \statespace.
\end{equation}

\subsection*{Interface conditions}
In this subsection, we introduce the necessary spaces for studying the inhomogeneous transmission condition \eqref{eq:TangentialJumps}. 
Let the space of square-integrable tangential vector fields on \(F\) be denoted with
\begin{equation*}
    \vvL^2_\tau(F) = \setc{\vvF\in L^2(F)^3}{\vvF\cdot \vvn = 0 \text{ a.e. in } F}.
\end{equation*}
In what follows, we identify this space with \(L^2(F)^2\), which is always possible with an appropriate change of origin, cf. Figure~\ref{fig:domain}.

It is well-known that the tangential trace and tangential component in \eqref{eq:NotationalSugarTang} can be extended to functions in \(\vvH(\curl)\) under minimal assumptions on the regularity of the domain.
See, e.g., \cite{Buf.C.S.2002}.
Nonetheless, they live in different subspaces of \(H^{-1/2}(F)^2\) and, thus, prohibit a clear interpretation of the impedance  transmission condition \eqref{eq:TangentialJumps}.
Circumventing this, one often switches to a regularized setting, where it is a priori assumed that the tangential trace and tangential component are \(L^2(F)^2\)-regular.
See, e.g., \cite{Ned.2001.AEE}.
We follow that practice and introduce the spaces with improved tangential regularity
\begin{subequations}\label{def:ImprovedRegularityCurlSpaces}
\begin{align}
    \vvH^\imp_0(\curl, Q) &= \setc{\vvV\in \vvH_0(\curl, Q)}{\vvV_\parallel \in L^2(F)^2},\\
    \vvH^\imp(\pcurl, Q) &= \setc{\vvV\in \vvH(\pcurl,Q)}{\jmp{\vvV_\tau} \in L^2(F)^2}.
\end{align}
\end{subequations}
They are both equipped with their graph norm.
With those spaces, we can interpret \eqref{eq:TangentialJumps} in the sense of \(L^2(F)^2\).
We refer to \cite[Sec.~3]{Chi.C.2023} for a detailed discussion of this practice.

The energy space is defined as
\begin{equation}\label{eq:EnergySpace}
    \energyspace = \vvV^\imp_{N,\epsilon}(Q) \times \vvV^\imp_{T,\mu}(Q),
    \qquad 
    \bigl\|\bigl(\vvE,\vvH\bigr)\bigr\|_{\energyspace}^2 
        = \bigl\|\vvE\bigr\|_{\vvH^\imp_0(\curl,Q)}^2
            + \bigl\|\vvH\bigr\|_{\vvH^\imp(\pcurl,Q)}^2,
\end{equation}
with the divergence-free spaces
\begin{subequations}\label{eq:VMaxwellSpaces}
\begin{align}
    \vvV^\imp_{N,\theta}(Q) &= \vvH^\imp_0(\curl,Q)\cap \vvH(\pdiv_\theta = 0,Q),\\
    \vvV^\imp_{T,\theta}(Q) &= \vvH^\imp(\pcurl, Q) \cap \vvH_0(\div_\theta = 0,Q).
\end{align}
\end{subequations}
Section \ref{sec:FunctionSpacesStudy} is devoted to the study of basic properties of those spaces, like Weber inequalities and density results.
Theorem~\ref{thm:HDensity} and \ref{thm:EDensity} will imply the density of \(PH^1(Q)^6\cap \energyspace\) in \(\energyspace\) with respect to the energy norm.
This fact is used in the proof of the next integration by parts formula.
It extends the usual formula in \(H^1\) from Theorem~2.2.28 of \cite{Ass.C.L.2018.MFCE}.
\begin{thm}\label{cor:PartialIntegrationFormula}
    For all \(\bigl(\vvE,\vvH\bigr) \in \energyspace\) it holds
    \begin{equation}\label{cor:PartialIntegrationFormula:Formula}
    \bigl(\vvE, \pcurl\vvH\bigr)_{L^2(Q)^3}
        - \bigl(\curl\vvE, \vvH\bigr)_{L^2(Q)^3}
    = \bigl(\vvE_{\parallel}, \jmp{\vvH_\tau} \bigr)_{L^2(F)^2}.
    \end{equation}
\end{thm}
\begin{proof}
    Let \(\bigl(\vvE,\vvH\bigr) \in PH^1(Q)^6\cap\energyspace\).
    For \(i\in \monosetc{1,2}\), Theorem~2.2.28 in \cite[Sec.~2.2.2]{Ass.C.L.2018.MFCE} shows
    \begin{equation*}
        \bigl(\vvE^{(i)},\curl \vvH^{(i)}\bigr)_{L^2(Q_i)^3}
            -   \bigl(\curl \vvE^{(i)}, \vvH^{(i)}\bigr)_{L^2(Q_i)^3}
        = - \int_{\partial Q_i} \restrict{\vvE^{(i)}}{\partial Q_i}\cdot \restrict{\bigl(\vvH^{(i)}\times \vvnu^{(i)}\bigr)}{\partial Q_i} \,\mathrm{d}\sigma
    \end{equation*}
    as \(E\in\vvH_0(\curl,Q)\).
    Moreover, the tangential trace of \(\vvE\) is continuous on \(F\) by \eqref{eq:ConnectionToFullSpaces:Curl}.
    Hence, \eqref{cor:PartialIntegrationFormula:Formula} follows for such fields.
    The general case then follows by approximation.
    See Theorem~\ref{thm:HDensity} and \ref{thm:EDensity} below for the needed density.
\end{proof}

\subsection*{\Wellposedness{}}
We are in the position to write \eqref{eq:PDE} as a linear evolution equation.
The Maxwell operator is defined as
\begin{subequations}\label{eq:MaxwellOperator}
\begin{equation}\label{eq:MaxwellOperatorDefiniton}
    M\,:\, \statespace \supset D(M) \to \statespace,
    \quad
    \bigl(\vvE,\vvH\bigr) \mapsto \bigl(\epsilon^{-1}\pcurl\vvH,-\mu^{-1}\curl \vvE\bigr),
\end{equation}
where the inhomogeneous transmission condition \eqref{eq:TangentialJumps} is included in the domain
\begin{equation}\label{eq:MaxwellOperatorDomain}
    D(M) = \setc[\big]{\bigl(\vvE,\vvH\bigr) \in \energyspace}{\jmp{\vvH_\tau} = -\sigma \vvE_\parallel}.
\end{equation}
Note that \(M\) is well-defined since \(\range(\pcurl) \subset \kernel(\pdiv)\), and \(\curl\bigl(\vvH_0(\curl,Q)\bigr) \in \vvH_0(\div,Q)\).
\end{subequations}
Thus, problem \eqref{eq:PDE} for \(\vvu(t) = \bigl(\vvE(t),\vvH(t)\bigr)\) reads as
\begin{equation}\label{eq:EvolutionEquation}
    \mathrm{d}_t \vvu(t) = M\vvu(t), \quad t>0,
    \qquad
    \vvu(0) = \vvu_0 \in \statespace.
\end{equation}
We prove unique existence of solutions to this problem with the Lumer-Phillips theorem.
See, for example, \cite[Cor.~II.3.20]{Eng.N.2011.SCOS}.
As a direct consequence of Theorem~\ref{cor:PartialIntegrationFormula}, we obtain that the Maxwell operator is dissipative.
\begin{cor}\label{cor:OperatorIsDissipative}
    The Maxwell operator \eqref{eq:MaxwellOperator} is dissipative, i.e., it satisfies
    \begin{equation*}
        \Bigl(M\bigl(\vvE,\vvH\bigr),\bigl(\vvE,\vvH\bigr)\Bigr)_{\statespace} = -\sigma\bigl\|\vvE_\parallel\bigr\|_{L^2(F)^2}^2 \leq 0,
        \qquad \text{for } \bigl(\vvE,\vvH\bigr) \in D(M).
    \end{equation*}
\end{cor}

\noindent Therefore, the rest of this section is devoted to proving that the Maxwell operator \eqref{eq:MaxwellOperator} is maximal.
We solve a suitable variational problem to obtain a magnetic field, given a datum in \(\statespace\).
By subsequently testing the variational problem with suitable test functions, we show that the solution satisfies the conditions defining the domain of the Maxwell operator \eqref{eq:MaxwellOperatorDomain}.
\begin{lem}
    The operator \((I-M)\,:\, D(M) \to \statespace\) is surjective.
\end{lem}
\begin{proof}
    Let \(\bigl(\vvF,\vvG\bigr) \in \statespace\).
    We define the bilinear form
    \(a\,:\,  \vvV^\imp_{T,\mu}(Q)  \times  \vvV^\imp_{T,\mu}(Q)  \to \mathbb{R}\) by
    \begin{equation}
        a\bigl(\vvH,\vvV\bigr) = 
        \bigl(\mu\vvH,\vvV\bigr)_{L^2(Q)^3}
        + \bigl(\epsilon^{-1}\pcurl\vvH,\pcurl\vvV\bigr)_{L^2(Q)^3}%
        + \sigma^{-1}\bigl(\jmp{\vvH_{\tau}}, \jmp{\vvV_{\tau}}\bigr)_{L^2(F)^2}.
    \end{equation}
    The bilinear form \(a\) is continuous and coercive.
    Let \(\ell\,:\,  \vvV^\imp_{T,\mu}(Q)  \to \mathbb{R}\) be the continuous linear form defined by
    \begin{equation*}
        \ell(\vvV) = \bigl(\mu\vvG,\vvV\bigr)_{L^2(Q)^3} - \bigl(\vvF,\pcurl \vvV\bigr)_{L^2(Q)^3}.
    \end{equation*}
    The Lax-Milgram lemma thus provides a unique solution of the variational problem
    \begin{equation}\label{eq:VariationalFormulation}
        \text{seek } \vvH\in  \vvV^\imp_{T,\mu}(Q)  \text{ s.t.}\quad a(\vvH,\vvV) = \ell(\vvV),\quad \text{for all } \vvV\in   \vvV^\imp_{T,\mu}(Q).
    \end{equation}

    Given the solution \(\vvH\in  \vvV^\imp_{T,\mu}(Q)\), we define the electric field by
    \begin{equation}\label{eq:ProofDefinition1}
        \vvE = \vvF +\epsilon^{-1}\pcurl \vvH.
    \end{equation}
    Note that \(\vvE\) belongs to  \(\vvH(\pdiv_\epsilon = 0,Q)\).

    We show that \(\vvE\) exhibits a weak variational curl.
    Let \(\vvU \in C^\infty(\overline{Q})^3\) and \(\vvP_{\mu,0}\,:\, L^2(Q)^3\to \vvH_0(\div_\mu = 0,Q)\) be the orthogonal projection with respect to \(\bigl(\mu\cdot,\cdot\bigr)_{L^2(Q)^3}\).
    Lemma~\ref{lem:HodgeProjectors}~\ref{lem:HodgeProjectors:HdivHomo} yields \(\vvP_{\mu} \vvU \in \vvH(\curl,Q)\subset \vvH(\pcurl,Q)\) and \(\jmp{(\vvP_{\mu,0}\vvU)_\tau} = 0\).
    Hence, \(\vvP_{\mu,0}\vvU \in \vvV_{T,\mu}^\imp(Q)\) and thus we can insert it into \eqref{eq:VariationalFormulation} to obtain
    \begin{multline}\label{eq:TestEquation2}
        \bigl(\mu\vvH,\vvP_{\mu,0}\vvU\bigr)_{L^2(Q)^3}
        + \bigl(\epsilon^{-1}\pcurl\vvH,\pcurl \vvP_{\mu,0}\vvU\bigr)_{L^2(Q)^3}\\
        =\bigl(\mu\vvG,\vvP_{\mu,0}\vvU\bigr)_{L^2(Q)^3} - \bigl(\vvF,\pcurl \vvP_{\mu,0}\vvU\bigr)_{L^2(Q)^3}.
    \end{multline}
    With the definition of \(\vvE\), Lemma~\ref{lem:HodgeProjectors}~\ref{lem:HodgeProjectors:HdivHomo}, and \(\vvG,\vvH \in \vvH_0(\div_\mu = 0,Q)\), we conclude that \eqref{eq:TestEquation2} is equivalent to
    \begin{equation*}
        \bigl(\vvE,\curl\vvU\bigr)_{L^2(Q)^3} 
            = \bigl(\mu\vvG - \mu \vvH,\vvP_{\mu,0}\vvU\bigr)_{L^2(Q)^3}
            = \bigl(\mu\vvG - \mu \vvH,\vvU\bigr)_{L^2(Q)^3}.
    \end{equation*}
    Since \(\vvU\) was chosen arbitrarily, \(\vvE\)  belongs to \(\vvH_0(\curl,Q)\), and
    \begin{equation}\label{eq:ProofDefinition2}
        \vvH = \vvG - \mu^{-1}\curl\vvE.
    \end{equation}
    It remains to prove that \(\vvE\) and \(\vvH\) satisfy the transmission condition posed in the domain of the Maxwell operator \eqref{eq:MaxwellOperatorDomain}.
    Therefore, take \(\vvU\in L^2(Q)^3\) with \(\vvU^{(i)}\in C^\infty(\overline{Q_i})^3\), \(i=1,2\).
    By Lemma~\ref{lem:HodgeProjectors}~\ref{lem:HodgeProjectors:HdivHomo} the field \(\vvP_{\mu,0}\vvU\) is contained in \(\vvH(\pcurl,Q)\) and hence \eqref{eq:VariationalFormulation} leads to
    \begin{equation*}
        \bigl(\mu\vvH - \mu\vvG, \vvP_{\mu,0}\vvU\bigr)_{L^2(Q)^3}
            + \bigl(\varepsilon^{-1}\pcurl\vvH + \vvF, \pcurl(\vvP_{\mu,0}\vvU)\bigr)_{L^2(Q)^3}
        = -\sigma^{-1} \bigl(\jmp{\vvH_\tau},\jmp{\vvP_{\mu,0}\vvU_\tau}\bigr)_{L^2(Q)^3}.
    \end{equation*}
    From Lemma~\ref{lem:HodgeProjectors}~\ref{lem:HodgeProjectors:HdivHomo}, \eqref{eq:ProofDefinition1}, and \eqref{eq:ProofDefinition2}, we then deduce
    \begin{equation}\label{eq:PoofEquation04}
        \bigl(\vvE,\pcurl \vvU\bigr)_{L^2(Q)^3}
            - \bigl(\curl\vvE,\vvU\bigr)_{L^2(Q)^3}
        = -\sigma^{-1}\bigl(\jmp{\vvH_\tau},\jmp{\vvU_\tau}\bigr)_{L^2(Q)^2}.
    \end{equation}
    On the other hand, integrating by parts in \(\vvH(\curl,Q_i)\times H^1(Q_i)\), the left-hand side can be written as
    \begin{equation}\label{eq:ProofEquation05}
        \bigl(\vvE,\pcurl \vvU\bigr)_{L^2(Q)^3} - (\curl \vvE,\vvU)_{L^2(Q)^3}
            = \langle\vvE_\parallel, \jmp{\vvU_\tau}\rangle_{1/2},
    \end{equation}
    where the angle-brackets denote the dual pairing in \(H^{1/2}(F)^2\). 
    (See the comments after \eqref{eq:ConnectionToFullSpaces}.)
    Equations \eqref{eq:PoofEquation04} and \eqref{eq:ProofEquation05} now imply \(-\sigma \vvE_\parallel = \jmp{\vvH_\tau}\) in the sense of \(H^{-1/2}_{00}(F)^2\).
    As \(\jmp{\vvH_\tau}\in L^2(F)^2\) by assumption, the trace \(\vvE_\parallel\) also belongs to \(L^2(F)^2\).
    Therefore, we found \(\bigl(\vvE,\vvH\bigr) \in D(M)\) such that \((I-M)\bigl(\vvE,\vvH\bigr) = \bigl(\vvF,\vvG\bigr)\), which proves the claim.
\end{proof}
The Lumer-Phillips theorem, see, e.g., \cite[Cor.~II.3.20]{Eng.N.2011.SCOS}, immediately provides the following result.
\begin{thm}\label{thm:Existence}
    \begin{enumerate}[label=\normalfont (\alph*)]
        \item\label{thm:Existence:Generator}
            The Maxwell operator \eqref{eq:MaxwellOperator} is the generator of a \(C^0\)-semigroup \(\bigl(e^{tM}\bigr)_{t\geq 0}\) of contractions.
        \item\label{thm:Existence:Existence}
            The evolution equation \eqref{eq:EvolutionEquation} has a unique mild solution 
            \begin{equation*}
                \vvu\in C^0\bigl([0,\infty),\statespace\bigr),\qquad \vvu(t) = e^{tM}\vvu_0,\quad t\geq 0.
            \end{equation*}
            For \(\vvu_0\in D(M)\), the solution is strong, i.e., \(\vvu \in C^1\bigl([0,\infty); \statespace\bigr)\cap C^0\bigl([0,\infty); D(M)\bigr)\).
    \end{enumerate}
\end{thm}

\subsection*{Surface continuity equation}
For \(\vvE(t) \in \vvH(\pdiv_\epsilon, Q)\), the normal traces \(\epsilon^{(i)}\vvE^{(i)}(t)\cdot\vvnu^{(i)} \in H^{-1/2}(\partial Q_i)\) are well-defined, see \cite[Thm.~2.2.18]{Ass.C.L.2018.MFCE}.
Their restrictions to \(F\) belong to \(H^{-1/2}_{00}(F)\).
We define the surface charge density and surface current as
\[
    \rho_F(t) = \jmp{\epsilon\vvE_\perp(t)} \in H^{-1/2}_{00}(F),
    \qquad
    \vvJ_F(t) = -\sigma\vvE_\parallel(t) \in \vvL^2_\tau(F),
    \qquad t \geq 0.
\]
As in the bulk, there is a hidden conservation equation, relating the surface charge density \(\rho_F\) and the surface current \(\vvJ_{F} = -\sigma \vvE_{\parallel}\).
Compare Section I.4.2.2 in \cite{Dau.Lio.1990a} and Lemma~8.1 in \cite{Sch.Spi.2022}.
\begin{thm}\label{thm:SurfaceChargeContinuity}
    Let \(\vvu_0\in\statespace\) and \(\vvu(t) = e^{tM}\vvu_0\) be the mild solution of \eqref{eq:EvolutionEquation}.
    Then the surface charge density \(\rho_F\) belongs to \(C^0\bigl([0,\infty),H^{-1/2}_{00}(F)\bigr)\) and satisfies
    \begin{equation}\label{eq:SurfaceChargeContinuity}
        \rho_F(t) = \rho_F(0) + \int_0^t \div_F\vvJ_F(s)\,\mathrm{d}s,
        \qquad\text{in } H^{-1/2}_{00}(F),\quad t\geq 0.
    \end{equation}
    
    For \(\vvu_0\in D(M)\), it additionally holds \(\rho_F\in C^1\bigl([0,\infty),H^{-1/2}_{00}(F)\bigr)\) and
    \begin{equation}\label{eq:DifferentialChargeContinuity}
        \partial_t\rho_F(t)
            = \div_F\vvJ_F(t),
        \qquad\text{in } H^{-1/2}_{00}(F),\quad t\geq 0.
    \end{equation}
\end{thm}
\begin{proof}
    Let first \(\vvu_0\in D(M)\).
    Along the trajectory, \(\epsilon^{(i)}\vvE^{(i)}(t)\) and its derivative \(\curl\vvH^{(i)}(t)\) are divergence-free, by the definition of \(\statespace\) and \(\range(\pcurl)\subset\kernel(\pdiv)\), so that 
    \[
    \epsilon^{(i)}\vvE^{(i)}\in C^1\bigl([0,\infty);\vvH(\div,Q_i)\bigr),\qquad i\in\monosetc{1,2}.
    \]
    By continuity of the normal trace and restriction to \(F\), we see that \(\rho_F\) belongs to \(C^1\bigl([0,\infty);H^{-1/2}_{00}(F)\bigr)\) with \(\partial_t\rho_F(t) = \jmp{(\pcurl\vvH(t))_\perp}\).
    As \(\vvH\in \vvH(\pcurl,Q)\), Theorem~3.15 and equation (34) of \cite{Buf.C.2001a} yield
    \begin{equation*}
        \jmp{(\pcurl\vvH(t))_{\perp}}
            = \div_F \jmp{\vvH_\tau(t)}
            = \div_F (-\sigma \vvE_\parallel(t))
            = \div_F \vvJ_F(t) \in H^{-1/2}_{00}(F),
    \end{equation*}
    where we used the transmission condition \eqref{eq:MaxwellOperatorDomain}.
    This proves the second claim.
    For \(\vvu_0\in\statespace\), the first claim follows by an approximation argument.
\end{proof}

\section{Function spaces and their properties}\label{sec:FunctionSpacesStudy}

\subsection*{Hodge projections}\label{sec:FunctionSpacesStudy:Hodge}
We adapt the Hodge projections to the present piecewise setting and follow the ideas in \cite{Nic.2000}.
The necessary results are collected in the following lemma.
\begin{lem}\label{lem:HodgeProjectors}
    Let \(\theta\) denote a piecewise constant and positive material parameter and equip \(L^2(Q)^3\) in the following with the inner product \((\theta\cdot, \cdot)_{L^2(Q)^3}\).
    \begin{enumerate}[label=\normalfont (\alph*)]
        \item\label{lem:HodgeProjectors:HdivHomo}
            Let \(\vvP_{\theta,0}\,:\, L^2(Q)^3\to \vvH_0(\div_\theta = 0,Q)\) denote the orthogonal projection.
            If \(\vvV\in L^2(Q)^3\) with \(\vvV^{(i)}\in C^\infty(\overline{Q_i})^3\) for \(i\in \{1,2\}\), then \(\vvP_{\theta,0} \vvV \in \vvH(\pcurl, Q)\) and
            \begin{equation*}
                \pcurl(\vvP_{\theta,0}\vvV) = \pcurl\vvV\,\, \text{in } Q, 
                \quad
                \jmp{(\vvP_{\theta,0}\vvV)_\tau} = \jmp{\vvV_\tau} \,\, \text{on } F. 
            \end{equation*}
        \item\label{lem:HodgeProjectors:PHdiv}
            Let \(\widetilde{\vvP}_{\theta}\,:\, L^2(Q)^3 \to \vvH(\pdiv_\theta = 0,Q)\) denote the orthogonal projection.
            If \(\vvV\in L^2(Q)^3\) with \(\vvV^{(i)}\in C^\infty(\overline{Q_i})^3\) for \(i\in \{1,2\}\), then \(\widetilde{\vvP}_{\theta} \vvV \in \vvH(\pcurl, Q)\) and
            \begin{equation*}
                \pcurl(\widetilde{\vvP}_{\theta}\vvV) = \pcurl \vvV\,\, \text{in } Q,
                \quad
                (\widetilde{\vvP}_{\theta}\vvV^{(i)}) \times \vvnu^{(i)} = \vvV^{(i)} \times \vvnu^{(i)}\,\, \text{on } \partial Q_i,\,\, i \in \monosetc{1,2}.
            \end{equation*}
    \end{enumerate}
\end{lem}
\begin{proof}
    We prove \ref{lem:HodgeProjectors:HdivHomo} in two steps.
            First, let \(\vvW\in L^2(Q)^3\) with \(\vvW^{(i)} \in C^\infty_c(Q_i)^3\).
            Since \(\vvW^{(i)}\) is compactly supported it satisfies \(\jmp{\vvW_\tau} = 0\) and thus
            \(
                \theta^{-1} \pcurl\vvW = \theta^{-1}\curl\vvW \in \vvH_0(\div_\theta = 0, Q).
            \)
            We conclude
            \begin{equation}\label{lem:HodgeProjectors:Hdiv:FirstEq}
                \bigl(\vvV, \pcurl \vvW\bigr)_{L^2(Q)^3}
                = \bigl(\theta \vvV, \theta^{-1}\pcurl\vvW\bigr)_{L^2(Q)^3}
                = \bigl(\vvP_{\theta,0} \vvV, \pcurl \vvW\bigr)_{L^2(Q)^3},
            \end{equation}
            which yields \(\vvP_{\theta,0}\vvV\in \vvH(\pcurl, Q)\) with \(\pcurl(\vvP_{\theta,0}\vvV) = \pcurl\vvV\) in \(Q\).

            Second, let \(\vvW\in \vvH_0(\curl,Q)\cap H^1(Q)^3\).
            Integration by parts on \(Q_i\) implies
            \begin{align*}
                (\pcurl\vvV, \vvW)_{L^2(Q)^3} 
                    - (\vvV, \curl\vvW)_{L^2(Q)^3} 
                    &= \langle \jmp{\vvV_\tau}, \vvW_\parallel\rangle_{1/2},\\
                (\pcurl(\vvP_{\theta,0}\vvV), \vvW)_{L^2(Q)^3} 
                    - (\vvP_{\theta,0}\vvV, \curl\vvW)_{L^2(Q)^3} 
                    &= \langle \jmp{(\vvP_{\theta,0}\vvV)_\tau}, \vvW_\parallel\rangle_{1/2},
            \end{align*}
            where the angle-brackets denote the dual paring in \(H^{1/2}(F)^2\).
            (See the comments after \eqref{eq:ConnectionToFullSpaces}.)
            The left-hand sides are equal after the first step
              and \eqref{lem:HodgeProjectors:Hdiv:FirstEq}, which proves that \(\jmp{(\vvP_{\theta,0}\vvV)_\tau} = \jmp{\vvV_\tau}\) on \(F\).

    The proof of statement \ref{lem:HodgeProjectors:PHdiv} follows along the same lines.
\end{proof}

\subsection*{Trace spaces \& right inverse}

In order to prove embedding and regularity results, we introduce the tangential component trace space associated with regular electric fields that satisfy the electric boundary conditions \eqref{eq:ElectricBoundaryConditions}. 
In the following we denote the facets of \(\partial Q\) with \(\Gamma_i^\pm\), uniquely defined by the outward normal vector \(\vve_i^\pm = \pm(\delta_{1i},\delta_{2i},\delta_{3i})\) for \(i=1,2,3\).
Additionally, we define the edges \(E^\pm_j = \operatorname{int}(\overline{F}\cap \overline{\Gamma^\pm_j})\) for \(j=2,3\) and the Lions-Magenes spaces by real interpolation
\begin{equation}\label{eq:PartialLionsMagenes}
    H^{1/2}_{E^\pm_j}(F) = \bigl[L^2(F), H^1_{E_j^\pm}(F)\bigr]_{1/2,2},
\end{equation}
where \(H^1_{E_j^\pm}(F) =\setc{f \in H^1(F)}{\restrict[]{f}{E_j^\pm} = 0}\).
The tangential trace space of regular vector fields \(\vvE\in \vvV_{N,\theta}^\imp \cap PH^1(Q)^3\) is given by
\begin{equation}\label{eq:VectorTraceSpace}
    \vvE_\parallel\in 
    \vvH^{1/2}_{\parallel,00}(F) 
        := H^{1/2}_{E_{3}^\pm}(F)\times H^{1/2}_{E_{2}^{\pm}}(F),
\end{equation}
and the trace map is continuous for the norm in \(PH^1(Q)^3\), 
see Lemma~7.3 in \cite{Zer.2021}.
The traces admit right-inverses which we gather in the following lemma.
\begin{lem}\label{lem:RightInverse}
    There exists linear and bounded operators
    \begin{subequations}
        \begin{align}
            \vvR_\parallel\,:\, \vvH^{1/2}_{\parallel,00}(F) &\to \vvH_0(\curl,Q)\cap \vvH(\div_\epsilon,Q)\cap H^1(Q)^3,
            \quad &&\text{with  } \bigl(\vvR_\parallel \vvV\bigr)_\parallel = \vvV,
            \label{lem:RightInverse:TangentialComponentTrace}\\
            \vvR_\tau\,:\, \vvH^{1/2}_{\parallel,00}(F) &\to \vvH(\pcurl,Q) \cap \vvH_0(\div_\mu,Q) \cap PH^{1}(Q)^3,
            \quad &&\text{with  } \jmp{(\vvR_\tau\vvV)_\tau} = \vvV.%
            \label{lem:RightInverse:TangentialTrace}
        \end{align}
    \end{subequations}
\end{lem}
The right inverse in \eqref{lem:RightInverse:TangentialComponentTrace} shows that the tangential component trace is surjective.
This is a well-known result, compare, for example Remark~6.7 in \cite{Buf.C.2001b} or Lemma~4.1 in \cite{Dor.Z.2023} for the specfic box geometry.
The second part in \eqref{lem:RightInverse:TangentialTrace} is non-standard and shows that we also find a representative of the rotated tangential component trace that is compatible with the conditions posed on the magnetic field.
We do not prove this result here since it is a variant of \cite[Lem.~4.1]{Dor.Z.2023}, and refer to the  detailed construction in Section~4 therein.

\subsection*{Tangential spaces}

In what follows, we adapt well-known results to our piecewise setting.
\begin{lem}\label{lem:HWeber}
    The estimate
    \begin{equation*}
        \|\vvH\|_{L^2(Q)^3} 
            \leq C
                \bigl( 
                    \|\pcurl \vvH\|_{L^2(Q)^3} + \|\jmp{\vvH_\tau}\|_{L^2(F)^2}
                \bigr)
    \end{equation*}
    is valid for all \(\vvH \in\vvV_{T,\theta}^\imp(Q)\) with a uniform constant \(C=C(\theta, Q) >0\).
\end{lem}
\begin{proof}
    Let \(\vvH \in\vvV_{T,\theta}^\imp(Q)\).
    There exists a unique vector potential \(\vvPhi \in PH^1(Q)^3\) such that
    \begin{equation}\label{eq:MagneticVectorPotential}
        \vvH = \frac{1}{\theta} \curl \vvPhi,
        \quad \vvPhi \in\vvH_0(\curl,Q)\cap\vvH(\div_\theta=0,Q).
    \end{equation}
    See, for example, \cite[Lem.~9.6]{Zer.2021}.
    We multipy \eqref{eq:MagneticVectorPotential} with \(\curl \vvPhi\) and integrate by parts in \(Q_i\) to obtain
    \begin{equation*}
        \bigl(\theta^{-1}\curl \vvPhi, \curl\vvPhi\bigr)_{L^2(Q)^3}
            = \bigl(\pcurl \vvH, \vvPhi\bigr)_{L^2(Q)^3} - \bigl(\jmp{\vvH_\tau}, \vvPhi_\parallel\bigr)_{L^2_\tau(F)^3}.
    \end{equation*}
    Cauchy-Schwarz and the trace inequality then imply
    \begin{equation*}
        \|\theta^{-1} \curl \vvPhi\|^2_{L^2(Q)^3}
            \leq C \|\vvPhi\|_{PH^1(Q)^3}
                \Bigl(\|\pcurl \vvH\|_{L^2(Q)^3} + \|\jmp{\vvH_\tau}\|_{L^2_\tau(F)^3}\Bigr),
    \end{equation*}
    where \(C = C(\theta, Q) > 0\).
    By \cite[Prop. 9.2]{Zer.2021}, the embedding \(\vvH_0(\curl,Q)\cap \vvH(\div_\theta= 0, Q) \hookrightarrow  PH^1(Q)^3\) is continuous, and \cite[Lemma 9.2]{Zer.2021} together with a compactness argument shows that   
   \(\|\theta^{-1} \curl \vvPhi\|^2_{L^2(Q)^3}\) is a norm on \(\vvH_0(\curl,Q)\cap\vvH(\div_\theta=0,Q)\) equivalent to its natural norm,
    which closes the proof. 
\end{proof}

\begin{thm}\label{thm:HDensity}
    The set \(PH^1(Q)^3\cap\vvV_{T,\theta}^\imp(Q)\) is dense in \(\vvV_{T,\theta}^\imp(Q)\) with respect to the graph norm.
\end{thm}
\begin{proof}
    Let \(\vvH\in\vvV_{T,\theta}^\imp(Q)\).
    By the proof of Lemma~\ref{lem:HWeber}, there exists a unique vector potential \(\vvPhi \in PH^1(Q)^3\) such that
    \begin{equation*}
        \vvH = \frac{1}{\theta} \curl \vvPhi,
        \quad \vvPhi \in\vvH_0(\curl,Q)\cap \vvH(\div_\theta = 0,Q).
    \end{equation*}
    Hence, \(\jmp{(\theta^{-1}\curl \vvPhi)_\tau}\) belongs to  \(L^2(F)^2\).
    Choose \(\bigl\{\widehat{\vvG}_n\bigr\}_{n\geq 0}\) in \(C_c^\infty(F)^2\subset\vvH^{1/2}_{\parallel,00}(F)\) such that
    \begin{equation*}
        \bigl\| \jmp{(\theta^{-1}\curl \vvPhi)_\tau} - \widehat{\vvG}_n \bigr\|_{L^2(F)^2}
            \to 0,
            \qquad
            n\to \infty.
    \end{equation*}

    Put \(\vvG = \jmp{\vvH_\tau}\) and let \(u_n\in H^1_0(F)\) solve
    \begin{equation}\label{eq:CorrectionWeakFormula}
        (\grad_F u_n,\grad_F v)_{L^2(F)^2}
            = (\vvG-\widehat{\vvG}_n,\grad_F v)_{L^2(F)^2},
        \qquad v\in H^1_0(F).
    \end{equation}
    Since \(\div_F\vvG = \jmp{(\pcurl\vvH)_\perp}\in H^{-1/2}_{00}(F)\) by \cite[(2.23)]{Ass.C.L.2018.MFCE}, elliptic regularity gives \(u_n\in H^{3/2}(F)\cap H^1_0(F)\).
    See the proof of Lemma~\ref{lem:FirstHiddenRegularity} below.
    Corollary~\ref{cor:TangentialRegularityLemma} then implies \(\grad_Fu_n\in\vvH^{1/2}_{\parallel,00}(F)\).
    Thus, \(\vvG_n = \widehat{\vvG}_n+\grad_Fu_n\) belongs to this trace space and satisfies
    \[
        \div_F\vvG_n = \div_F\vvG,
        \qquad
        \|\vvG_n-\vvG\|_{L^2(F)^2}
            \leq 2\|\widehat{\vvG}_n-\vvG\|_{L^2(F)^2}\to0,
    \]
    where the estimate follows by testing \eqref{eq:CorrectionWeakFormula} with \(u_n\).

    We construct a sequence of vector fields \(\monosetc{\vvPhi_n}_{n\geq 0}\) solving the elliptic transmission problem
    \begin{subequations}\label{eq:VectorEllipticTransmissionProblem}
        \begin{align}
            \curl \theta^{-1} \curl \vvPhi_n^{(i)} 
                &= \curl \theta^{-1} \curl \vvPhi^{(i)},
                &&\text{in } Q_i,\\
            \jmp{(\theta^{-1} \curl\vvPhi_n)_\tau} 
                &= \vvG_n,
                &&\text{on } F,\\
            \vvPhi_n \in \vvH_0&(\curl,Q)\cap \vvH(\div_\theta = 0,Q).
        \end{align}
    \end{subequations}
    Set \(\vvF = \pcurl\theta^{-1}\curl \Phi \in L^2(Q)^3\) and take \(W_N = \vvH_0(\curl,Q)\cap \vvH(\div_\theta = 0,Q)\).
    We define the forms
    \begin{align*}
        a \,:\, W_N\times W_N &\to \R,\qquad
        &a(\vvPhi,\vvPsi) &= \bigl(\theta^{-1}\curl \vvPhi, \curl \vvPsi\bigr)_{L^2(Q)^3},\\
        \ell \,:\, W_N &\to \R, \qquad
        &\ell(\vvPsi) &= \bigl(\vvF, \vvPsi\bigr)_{L^2(Q)^3} - \langle \vvPsi_\parallel, \vvG_n\rangle_{-1/2\times 1/2}.
    \end{align*}
    Both forms are continuous on \(W_N\), and the bilinear form \(a\) is coercive due to the Weber inequality proven in \cite[Lem.~9.4]{Zer.2021}.
    Hence, the Lax--Milgram lemma provides maps \(\vvPhi_n\) in \(W_N\) that solve \eqref{eq:VectorEllipticTransmissionProblem} weakly with test functions in \(W_N\).

    The compatibility \(\div_F\vvG_n = \jmp{\vvF_\perp}\) gives, for \(q\in H^1_0(Q)\),
    \begin{align*}
        \ell(\grad q)
            &= (\vvF,\grad q)_{L^2(Q)^3}
                - \langle\grad_F(q|_F),\vvG_n\rangle\\
            &= -\langle\jmp{\vvF_\perp},q|_F\rangle
                + \langle\div_F\vvG_n,q|_F\rangle\\
            &= \langle\div_F\vvG_n-\jmp{\vvF_\perp},q|_F\rangle = 0.
    \end{align*}
    Here Green's formula is applied in each \(Q_i\) and on \(F\), using \(\pdiv\vvF=0\).
    The Helmholtz decomposition \(\vvH_0(\curl,Q) = W_N\oplus\grad H^1_0(Q)\) from \cite[(6.16)]{Ass.C.L.2018.MFCE}, therefore extends the variational identity \(a(\vvPhi_n, \Psi) = \ell(\Psi)\) to all test fields \(\Psi\in \vvH_0(\curl,Q)\), yielding the bulk and interface equations in \eqref{eq:VectorEllipticTransmissionProblem}.

    Furthermore, \(\vvPhi_n\) belongs to  \(PH^1(Q)^3\) for \(n\geq 0\), since \(W_N \hookrightarrow PH^1(Q)^3\) continuously by Proposition~9.3 in \cite{Zer.2021}.
    
    We apply elliptic regularity to every component of \(\vvPhi_n = \bigl(\Phi_{1,n},\Phi_{2,n},\Phi_{3,n}\bigr)\) individually.
    Because of \(\vvPhi_n \in \vvH_0(\curl,Q)\), Corollary~3.1.16 of \cite{Ass.C.L.2018.MFCE} and \eqref{eq:ConnectionToFullSpaces:Div} imply that \(\curl\vvPhi_n\) belongs to \(\vvH_0(\div,Q)\).
    This allows to apply the vector-Laplace identity \(\curl \curl  = -\bm{\Delta} + \grad \div\) in \(H^1(Q_i)\), for \(i\in \monosetc{1,2}\), and we see that \(\Phi_{1,n}\) is the weak solution of
    \begin{subequations}\label{eq:PiecwieseEllipticEquation}
        \begin{align}
            -\theta\Delta \Phi_{1,n}^{(i)} &= \theta^2 F^{(i)}_1,
                \qquad &\text{in } Q_i,\\
            \jmp{\theta \Phi_{1,n}} &= 0,
                \qquad &\text{on } F,
                \label{eq:PiecwieseEllipticEquation:ThetaHomoDirichletJmp}\\
            \Phi_{1,n} &= 0,
                \qquad &\text{on } \Gamma_2^{\pm} \cup \Gamma_3^{\pm},
                \label{eq:PiecwieseEllipticEquation:HomoDirichlet}
        \end{align}
        where we use test functions \(\psi \in PH^1(Q)\) that satisfy \eqref{eq:PiecwieseEllipticEquation:ThetaHomoDirichletJmp} and \eqref{eq:PiecwieseEllipticEquation:HomoDirichlet}.
        Hence, Proposition~8.1 in \cite{Zer.2021} with right-hand side \(f = \theta^2F_1^{(i)}\) implies that \(\Phi_{1,n} \in PH^2(Q)\),  as well as
        \begin{align}
            \jmp{\partial_1 \Phi_{1,n}} &= 0,
                \qquad &\text{on } F,\\
             \partial_\vvnu \Phi_{1,n} &= 0,
                \qquad &\text{on } \Gamma_1^{\pm}.
        \end{align}
    \end{subequations}

    Analogously, the second component is the weak solution of
        \begin{subequations}
        \begin{align}
            -\Delta \Phi_{2,n}^{(i)} &= \theta F^{(i)}_2,
                \qquad &\text{in } Q_i,\\
            \jmp{\Phi_{2,n}} &= 0,
                \qquad &\text{on } F,\\
            \Phi_{2,n} &= 0,
                \qquad &\text{on } \Gamma_1^{\pm} \cup \Gamma_3^{\pm},\\
            \jmp{\theta^{-1}\partial_1 \Phi_{2,n}} &= \jmp{\theta^{-1}\partial_2 \Phi_{1,n}} - G_{1,n},
                \qquad &\text{on } F.
        \end{align}
        Here we use that due to \eqref{eq:VectorTraceSpace}, the map \(G_{1,n}\) is element of \(H^{1/2}_{E_3^\pm}(F)\) for \(n\geq 0\).
        Moreover, since \(\Phi_{1,n} \in PH^2(Q)\), Lemma~2.1 of \cite{Eil.Sch.2018} and \eqref{eq:PiecwieseEllipticEquation:HomoDirichlet} imply that \(\partial_2 \Phi_{1,n}\) is contained in \(PH^1_{\Gamma_3^\pm}(Q)\).
        Hence, \cite[Lem.~7.3]{Zer.2021} yields \(\jmp{\theta^{-1} \partial_2 \Phi_{1,n}} \in H^{1/2}_{E_3^\pm}(F)\).
        From the elliptic regularity result \cite[Lemm.~8.13]{Zer.2021} we deduce \(\Phi_{2,n} \in PH^2(Q)\) and the properties
        \begin{align}
            \partial_\vvnu \Phi_{2,n} &= 0,
                \qquad &\text{on } \Gamma_2^{\pm}.
        \end{align}
    \end{subequations}
    
    Repetition of the latter argument for the third component shows that \(\Phi_{3,n} \in PH^2(Q)\), and, therefore, that
    \(\vvPhi_n \in PH^2(Q)^3\) for \(n\geq 0\).

    Thus, the fields \(\vvH_n = \theta^{-1} \curl\vvPhi_n\) for \(n\geq 0\), belong to \(PH^1(Q)^3 \cap \vvV_{T,\theta}^\imp(Q)\) with
    \begin{align*}
        \|\pcurl \vvH - \pcurl \vvH_n\|_{L^2(Q)^3}
        \lesssim \sum_{i\in\monosetc{1,2}}
            \|\pcurl\theta^{-1}\curl(\vvPhi^{(i)} - \vvPhi_n^{(i)})\|_{L^2(Q_i)^3} &= 0,\, n \geq 0,\\
        \|\jmp{\vvH_\tau} - \jmp{(\vvH_{n})_\tau} \|_{L^2(F)^2} 
            = \bigl\| \jmp{(\theta^{-1}\curl \vvPhi)_\tau} - \vvG_n \bigr\|_{L^2(F)^2}
            &\to 0,
        \, n\to \infty.
    \end{align*}
    Lemma~\ref{lem:HWeber} then proves the claim.
\end{proof}

\subsection*{Normal spaces}

We establish similar results for the space \(\vvV_{N,\theta}^\imp(Q)\) as in the previous section.
The proofs follow along the same lines, but with a minor difference.
Whereas in the space \(\vvV_{T,\theta}^\imp(Q)\) the regularity of the discontinuous tangential jump is controlled, in \(\vvV_{N,\theta}^\imp(Q)\), the regularity of the continuous tangential component is controlled, and the discontinuous normal jump is, a priori, uncontrolled.
Therefore, we first extend a known result, \cf{} \cite[Thm.~2]{Cos.1990}, to our interface setting, that bounds the regularity of the normal jump in \(\vvV_{N,\theta}^\imp(Q)\) by the regularity of the tangential component.
\begin{lem}\label{lem:RellichInequality}
    \begin{enumerate}[label=\normalfont(\alph*)]
        \item \label{lem:RellichInequalitySquare}
            Let \(S = (0,1)^3\) be the unit cube, \(\Gamma = \monosetc{0}\times (0,1)^2\), and \(\Gamma^\ast = \partial S \setminus \Gamma\).
            If \(v\in H^1_{\Gamma^\ast}(S)\) with \(-\Delta v = 0\), then 
            \begin{equation*}
                \restrict[]{v}{\Gamma} \in H^1(\Gamma) \Leftrightarrow \partial_{\vvnu} v \in L^2(\Gamma),
                \qquad 
                C_1 \|\restrict[]{v}{\Gamma}\|_{H^1(\Gamma)}
                    \leq \| \partial_{\vvnu} {v}\|_{L^2(\Gamma)}
                    \leq C_2 \|\restrict[]{v}{\Gamma}\|_{H^1(\Gamma)},
            \end{equation*}
            with constants \(C_1,C_2 >0\) independent of \(v\).
        \item \label{lem:RellichInequalityInterface}
            Let \(v\in H^1_0(Q)\) with \(-\div(\theta^{(i)} \grad v^{(i)}) = 0\), for \(i \in \monosetc{1,2}\). 
            Then we obtain
            \begin{equation*}
                \restrict{v}{F} \in H^1(F)
                \Leftrightarrow
                \jmp{\theta \partial_\vvn v} \in L^2(F),
                \qquad
                \widetilde{C}_1 \| \restrict[]{v}{F}\|_{H^1(F)}
                    \leq \| \jmp{\theta \partial_{\vvn} v}\|_{L^2(\Gamma)}
                    \leq \widetilde{C}_2 \| \restrict[]{v}{F}\|_{H^1(F)},
            \end{equation*}
            with constants \(\widetilde{C}_1,\widetilde{C}_2 > 0\) independent of \(v\).
    \end{enumerate}
\end{lem}

\begin{proof}
    For (a). Since \(v\in H^1_{\Gamma^\ast}(S)\), Lemma~7.3 in \cite{Zer.2021} shows that \(\restrict{v}{\Gamma}\) belongs to  \(H^{1/2}_{00}(\Gamma)\). 
    Let \(D(-\Delta_\Gamma) = H^2(\Gamma)\cap H^1_0(\Gamma)\) be the Dirichlet-Laplacian on \(\Gamma\).
    It admits an orthonormal basis of eigenfunctions \(\Phi_{n,m}(y,z) = \sin(n\pi y) \sin(m\pi z)\), for \(n,m\geq 1\).
    Since \(-\Delta_\Gamma\) is self-adjoint and positive, the eigenfunctions also form a basis of \(D((-\Delta_\Gamma)^{1/4})\).
    By interpolation, \(H^{1/2}_{00}(\Gamma)\) is equivalent to \(D((-\Delta_\Gamma)^{1/4})\) with respect to norms, and we find an expansion
    \begin{equation*}
        \restrict{v}{\Gamma} = \sum_{n,m\geq 1} b_{nm} \Phi_{nm}, \,\,\,\,\text{  in } H^{1/2}_{00}(F),
    \end{equation*}
    with coefficients \(b_{nm} = (\restrict[]{v}{\Gamma}, \Phi_{nm})_{L^2(\Gamma)}\), for \(n,m\geq 1\).
    By separation of variables, we obtain the expansion
    \begin{equation*}
        v = \sum_{n,m\geq 1} b_{nm} X_{nm} \Phi_{nm}, \,\,\,\,\text{  in } H^{1}(S),
    \end{equation*}
    where 
    \begin{equation*}
        X_{nm}(x) = \frac{\sinh\bigl(\pi\sqrt{n^2+m^2}(1-x)\bigr)}{\sinh{\bigl(\pi \sqrt{n^2+m^2}\bigr)}},
        \quad n,m\geq 1.
    \end{equation*}
    Thus, we see that
    \begin{align*}
        \restrict{v}{\Gamma} \in H^1(\Gamma) 
            &\Leftrightarrow \sum_{n,m\geq 1} |b_{nm}|^2 \bigl(1+(n\pi)^2+(m\pi)^2\bigr) < \infty,\\
        \restrict{\partial_\nu v}{\Gamma} \in L^2(\Gamma)
            &\Leftrightarrow \sum_{n,m\geq 1} |b_{nm}|^2 \bigl((n\pi)^2 + (m\pi)^2\bigr)
                \coth\bigl(\pi\sqrt{n^2+m^2}\bigr)^2 < \infty.
    \end{align*}
        
    The equivalence now holds by the uniform boundedness of \(\coth\) and \(1/\coth\) on the interval \([\sqrt{2}\pi,\infty)\).

    For (b).
    By rescaling and localization, we can reduce to the case \(\zeta = 0\).
    We then apply part \ref{lem:RellichInequalitySquare} to the map \(\widetilde{v} (x,\cdot ) = \theta^{(1)}v^{(1)}(-x,\cdot) + \theta^{(2)}v^{(2)}(x,\cdot)\) for \(x\in (0,1)\), to derive the claim.
\end{proof}

\begin{cor}\label{cor:ETraceinequalities}
            For \(\vvV\in \vvV_{N,\theta}^\imp(Q)\) we have \(\theta^{(i)}\vvV^{(i)}_\perp \in L^2(F)\) for \(i\in \monosetc{1,2}\), with the estimates
            \begin{align*}
                \|\jmp{\theta\vvV_\perp}\|_{L^2(F)}
                    &\lesssim \|\vvV_\parallel\|_{L^2(F)^2} + \|\curl \vvV\|_{L^2(Q)^3},\\
                \|\vvV_\parallel \|_{L^2(F)^2}
                    &\lesssim \|\jmp{\theta\vvV_\perp}\|_{L^2(F)}
                        + \|\curl \vvV\|_{L^2(Q)^3}.
            \end{align*}
\end{cor}
\begin{proof}
    Let \(\vvV \in \vvV_{N,\theta}^\imp(Q)\).
    We construct a regular decomposition of \(\vvV\) in terms of a vector and a scalar potential. 
    Since \(\vvV \in \vvH_0(\curl,Q)\), the field \(\curl \vvV\) belongs to \(\vvH_0(\div= 0, Q)\).
    There thus exists a vector potential \(\vvPhi \in H^1(Q)^3\) with
    \begin{equation*}
        \curl\vvV = \curl\vvPhi,\quad \vvPhi\in\vvH_0(\curl,Q)\cap \vvH(\div =  0 ,Q),
        \quad \|\vvPhi\|_{H^1(Q)^3} \leq \|\curl \vvV\|_{L^2(Q)^3}.
    \end{equation*}
    See, e.g., \cite[Thm.~3.5.1]{Ass.C.L.2018.MFCE}.
    Furthermore, \(\curl(\vvV-\vvPhi) = 0\) so that there exists a scalar potential
    \begin{equation}\label{eq:DecompositionHhCurl}
        \vvV = \vvPhi + \grad \phi,\quad \phi\in H^1_0(Q),
    \end{equation}
    by \cite[Thm.~3.3.9]{Ass.C.L.2018.MFCE}.
    Taking piecewise the divergence in \eqref{eq:DecompositionHhCurl}, we see that \(-\Delta \phi^{(i)} = 0\) on \(Q_i\).
    By \eqref{eq:VectorTraceSpace}, the trace \(\vvPhi_\parallel\) is contained in \(\vvH^{1/2}_{\parallel,00}(F)\), and thus \(\restrict{\phi}{F}\) in \(H^1(F)\) because of \eqref{eq:DecompositionHhCurl} and \(V\in \vvV_{N,\theta}^\imp\).
    An application of Lemma~\ref{lem:RellichInequality}~\ref{lem:RellichInequalityInterface} deduces the assertion.
\end{proof}

\begin{lem}\label{lem:EWeber}
    The estimate
    \begin{equation*}
        \|\vvE\|_{L^2(Q)^3} 
            \leq C
                \bigl( 
                    \|\curl \vvE\|_{L^2(Q)^3} + \|\vvE_\parallel\|_{L^2(F)^2}
                \bigr)
    \end{equation*}
    is valid for all \(\vvE \in \vvV_{N,\theta}^\imp(Q)\) with a uniform constant \(C>0\).
\end{lem}
\begin{proof}
    Let \(\vvE \in \vvV_{N,\theta}^\imp(Q)\), and \(\phi\in H^1_0(Q)\) be the weak solution of the elliptic transmission problem
    \begin{equation}\label{eq:ETransmissionProblem}
    \begin{aligned}
        -\div\bigl(\theta^{(i)}\grad \phi^{(i)}\bigr) &= 0, &&\text{in } Q_i,\\
        \jmp{\phi} &= 0, &&\text{on } F,\\
        \jmp{\theta \partial_\vvn \phi} &= \jmp{\theta \vvE_\perp}, &&\text{on } F,\\
        \phi &= 0, &&\text{on } \partial Q.
    \end{aligned}
    \end{equation}
    Testing \eqref{eq:ETransmissionProblem} with \(\phi\) we derive the inequality \(\|\phi\|_{H^1(Q)} \lesssim \|\jmp{\theta \vvE_\perp}\|_{L^2(F)}\) in a standard way.

    By construction \(\widetilde{\vvE} = \vvE-\grad\phi \) belongs to \(\vvH_0(\curl,Q)\cap \vvH(\div_\theta = 0, Q)\), which embeds continuously into \(PH^1(Q)^3\), see for example \cite[Prop.~9.3]{Zer.2021}.
    Then, we obtain
    \begin{equation*}
        \|\vvE\|_{L^2(Q)^3} \leq \|\widetilde{\vvE}\|_{L^2(Q)^3}
                + \|\grad \phi\|_{L^2(Q)^3}
            \lesssim \|\curl \widetilde{\vvE}\|_{L^2(Q)^3}
                + \|\jmp{\theta \vvE_{\perp}}\|_{L^2(F)},
    \end{equation*}
    where we used \cite[Thm.~6.1.4]{Ass.C.L.2018.MFCE} to estimate the divergence-free part \(\widetilde{\vvE}\).
    The proof follows now from \(\curl\vvE = \curl\widetilde{\vvE}\) and Corollary~\ref{cor:ETraceinequalities}.
\end{proof}

\begin{thm}\label{thm:EDensity}
    The set \(PH^1(Q)^3\cap \vvV_{N,\theta}^\imp(Q)\) is dense in \(\vvV_{N,\theta}^\imp(Q)\) with respect to the graph norm.
\end{thm}
\begin{proof}
    Let \(\vvE \in \vvV_{N,\theta}^\imp(Q)\), and \(\phi\in H^1_0(Q)\) be the weak solution of \eqref{eq:ETransmissionProblem}.
    As in the proof above, there exists \(\widetilde{\vvE} \in \vvH_0(\curl,Q)\cap \vvH(\div_\theta = 0, Q) \hookrightarrow PH^1(Q)^3\) such that \(\vvE = \widetilde{\vvE} + \grad \phi\).
    We construct an approximation on \(\phi\) with higher regularity.
    Let \(\monosetc{g_n}_{n\geq 0} \subset H^{1/2}_{00}(F)\) be such that
    \begin{equation*}
        \|\jmp{\theta\vvE_\perp} - g_n\|_{L^2(F)} \to 0,
        \quad
        n \to \infty.
    \end{equation*}
    For each \(n \geq 0\), the Lax-Milgram lemma provides a weak solution \(\phi_n\in H^1_0(Q)\)
    of 
    \begin{equation*}
    \begin{aligned}
        -\div\bigl(\theta^{(i)}\grad \phi^{(i)}_n\bigr) &= 0, &&\text{in } Q_i,\\
        \jmp{\phi_n} &= 0, &&\text{on } F,\\
        \jmp{\theta \partial_\vvn \phi_n} &= g_n, &&\text{on } F,\\
        \phi_n &= 0, &&\text{on } \partial Q.
    \end{aligned}
    \end{equation*}
    We apply a straightforward variant of Lemma~8.14 of \cite{Zer.2021} with \(\eta = \theta^{-1}\) to the scaled function \(\psi_n = \theta\phi_n\), which implies \(\psi_n \in PH^2(Q)\cap PH^1_0(Q)\) with \(\jmp{\theta^{-1} \psi_h} = 0\).
    Hence, \(\phi_n\) is contained in \(PH^2(Q) \cap H^1_{0}(Q)\).

    We now set \(\vvE_n = \widetilde{\vvE} + \grad \phi_n\) for \(n\geq 0\).
    Note that \(\grad \phi_n \in \vvH_0(\curl, Q)\) by \eqref{eq:ConnectionToFullSpaces:Curl}, \(\jmp{\phi_n}= 0\) on \(F\), and \(\phi_n = 0\) on \(\partial Q\).
    Hence, \(\vvE_n\) belongs to  \(PH^1(Q) \cap \vvV_{N,\theta}^\imp(Q)\) and satisfies
    \begin{equation*}
        \|\curl \vvE - \curl \vvE_n \|_{L^2(Q)^3}
            = \|\curl \widetilde{\vvE} - \curl \widetilde{\vvE}\|_{L^2(Q)^3} = 0,
            \quad n\geq 0.
    \end{equation*}
    Furthermore, Lemma~\ref{lem:RellichInequality}~\ref{lem:RellichInequalityInterface} and Corollary~3.1.16 in \cite{Ass.C.L.2018.MFCE} yield
    \begin{equation*}
        \|\vvE_\parallel - (\vvE_n)_\parallel\|_{L^2(F)^2}
            \leq \|\restrict{\phi}{F} - \restrict{\phi_n}{F}\|_{H^1(F)}
            \leq \|\jmp{\theta\vvE_\perp} - g_n\|_{L^2(F)} \to 0,\quad n\to\infty. 
    \end{equation*}
    This proves the claim together with Lemma~\ref{lem:EWeber}.
\end{proof}

\section{Regularity}
\label{sec:Regularity}

Although Maxwell's equations with impedance boundary conditions are not the focus of this work, we recall that in that setting the domain of the Maxwell operator exhibits hidden regularity, see, for example, \cite{Cos.N.D.2010}. In our problem, we impose interface conditions of a similar type, which allows us to adapt these techniques to the interface case. A careful analysis is nevertheless required to verify that the relevant compatibility conditions at the edges of the interface are fully satisfied.

The following lemma shows that a function with a homogeneous Dirichlet trace on the square admits tangential derivatives that remain extendable by zero. Such a result is proven in Lemma~3.8 of \cite{Buf.C.2001a}.
For completeness and greater clarity, a self-contained proof is presented below.

\begin{lem}\label{lem:TangentialRegularityLemma}
    Let \(S = \setc{(x_1,x_2)\in \mathbb{R}^2}{|x_1| < 1, |x_2| < 1}\) and \(S_j = \setc{x\in \overline{S}}{x_j = \pm 1}\) for \(j\in \monosetc{1,2}\).
    If \(f\in H^{3/2}(S) \cap H^1_0(S)\),
    then \(\partial_{x_2}f \in H^{1/2}_{S_1}(S)\), \(\partial_{x_1}f\in H^{1/2}_{S_2}(S)\), and 
    \[
        \|\partial_{x_2} f\|_{H^{1/2}_{S_1}(S)} + \|\partial_{x_1} f\|_{H^{1/2}_{S_2}(S)} \lesssim \|f\|_{H^{3/2}(S)}.
    \]
\end{lem}
\noindent 
Here, \(H^{1/2}_{S_j}(S) = \bigl[L^2(S),H^1_{S_j}(S)\bigr]_{1/2,2}\) denotes the Lions-Magenes space defined by real interpolation, as in \eqref{eq:PartialLionsMagenes}.
\begin{proof}[Proof of Lemma~\ref{lem:TangentialRegularityLemma}]
In this proof we set
\begin{equation*}
    \grad_2 f = (\partial_{x_1} f, \partial_{x_2} f),
    \qquad \div_2 \vvv = \partial_{x_1} v_1 + \partial_{x_2}v_2,
    \qquad \scurl_2 \vvv = \scurl_2 (v_1,v_2) = \partial_{x_1} v_2 - \partial_{x_2} v_1.
\end{equation*}
Let \(\vvu = (u_1, u_2) = \grad_{2} f\).
Then \(\vvu \in \kernel(\scurl_2) \subset \vvH(\scurl_2,S)\) and the tangential trace mappings \(\vvu \mapsto \restrict{u_2}{S_1}\) and \(\vvu \mapsto \restrict[]{u_1}{S_2}\) are well-defined in the sense of \(H^{-1/2}_{00}(S_1)\), and \(H^{-1/2}_{00}(S_2)\), respectively. (Compare the discussion after \eqref{eq:ConnectionToFullSpaces}.)
For \(s\in [0,1]\) we define the set
\begin{equation*}
    V^s = \setc{\vvv \in H^{s}(S)^2}{\scurl_2 \vvv = 0,\, \restrict{v_2}{S_1} = 0 =\restrict[]{v_1}{S_2}}.
\end{equation*}
The space \(V^s\) is a closed subspace of \(\vvH^s(\scurl_2, S) = \setc{\vvv\in H^s(S)^2}{\scurl_2 \vvv \in H^s(S)}\).
Furthermore, we define the Hodge mapping
\[
    P = \grad_2 (\Delta_D)^{-1}\div_2\,:\, \vvH^s(\scurl_2, S) \to V^s,
\] where \((\Delta_D)^{-1}\,:\, H^{s-1}(S) \to H^{s+1}(S)\cap H^1_0(S)\) denotes the solution operator of the homogeneous Dirichlet problem on \(S\).
The operator \(P\) is a linear and bounded projection on \(V^s\),  cf. \cite[Thm.~3.3.9]{Ass.C.L.2018.MFCE}.

Then for any \(\theta\in (0,1)\), we obtain with the invariance of interpolation functors under retraction, see \cite[Thm~1.17.1.1]{Tri.1978.ITFS}, and well-known interpolation results for \(\vvH(\curl)\) spaces, see 
\cite[Thm.~1]{Hip.L.Z.2012}, that
\begin{align*}
    \bigl[V^0, V^1\bigr]_{\theta,2}
        = P \bigl[\vvH^0(\scurl_2, S),\vvH^1(\scurl_2, S)\bigr]_{\theta,2}
        = P \bigl(\vvH^{\theta}(\scurl_2, S)\bigr)
        = V^{\theta}.
\end{align*}
Finally, the identification of components \(T\vvv = (v_1, v_2)\) is bounded in the sense of \(V^0 \to L^2(S) \times L^2(S)\) and \(V^1 \to H^1_{S_2}(S) \times H^1_{S_1}(S)\).
Therefore, by interpolation, \(T\) is bounded from \(V^{\theta} \to \bigl[L^2(S), H^1_{S_2}(S)\bigr]_{\theta,2} \times \bigl[L^2(S), H^1_{S_1}(S)\bigr]_{\theta,2}\), which proves the claim.
\end{proof}
As a direct consequence, we obtain the following result for functions on the interface.
\begin{cor}\label{cor:TangentialRegularityLemma}
    Let \(\phi \in H^{3/2}(F)\cap H^1_0(F)\), then  \(\grad_F \phi \in \vvH^{1/2}_{\parallel,00}(F)\) and \(\|\grad_F\phi\|_{\vvH^{1/2}_{\parallel,00}(F)}\lesssim \|\phi\|_{H^{3/2}(F)}\).
\end{cor}

The following lemma identifies the hidden boundary regularity of fields in the domain of the Maxwell operator \eqref{eq:MaxwellOperatorDefiniton}.
\begin{lem}\label{lem:FirstHiddenRegularity}
    Let \(\bigl(\vvE,\vvH\bigr) \in D(M)\). Then \(\jmp{\epsilon\vvE_\perp}\) belongs to \(H^{1/2}_{00}(F)\) and \(\jmp{\vvH_\tau}\) to \(\vvH^{1/2}_{\parallel,00}(F)\).
    Furthermore, we control
    \begin{equation}\label{eq:JumpToNormControl}
        \|\jmp{\epsilon\vvE_\perp}\|_{H^{1/2}_{00}(F)} 
        + \|\jmp{\vvH_\tau}\|_{\vvH^{1/2}_{\parallel,00}(F)}
            \lesssim \bigl\|\bigl(\vvE,\vvH\bigr)\bigr\|_{\energyspace}.
    \end{equation}
\end{lem}
\begin{proof}
    1).
    Let \(\vvH_{\partial Q_1\setminus F}(\curl,Q_1) = \setc{\vvV\in \vvH(\curl,Q_1)}{\restrict{\vvV\times \nu}{\partial Q_1\setminus F} = 0 }\), see \cite[p. 99]{Ass.C.L.2018.MFCE}.
    We choose a regular decomposition \(\vvPhi\in H^1(Q_1)^3\cap \vvH_{\partial Q_1\setminus F}(\curl,Q_1)\) and \(\phi \in H^1_{\partial Q_1\setminus F}(Q_1)\) such that 
    \begin{equation}\label{eq:FirstDecomposition}
       \vvE^{(1)} = \vvPhi + \grad \phi,
        \quad \|\vvPhi\|_{H^1(Q_1)^3} + \|\phi\|_{H^1(Q_1)} \lesssim \|\vvE^{(i)}\|_{\vvH(\curl,Q_1)}.
    \end{equation}
    See, for example, \cite[Thm.~5.2]{Pau.S.2022}.
    As in \eqref{eq:VectorTraceSpace}, we have \(\vvPhi_\parallel \in \vvH^{1/2}_{\parallel,00}(F)\), and Lemma~7.3 in \cite{Zer.2021} shows that \(\restrict{\phi}{F} \in H^{1/2}_{00}(F)\).
    Since \(\vvE_{\parallel} \in L^2(F)^2\), the map \(\grad_F \restrict[]{\phi}{F}  = \vvE_\parallel - \vvPhi_\parallel\) belongs to \(L^2(F)^2\) and therefore \(\restrict[]{\phi}{F}\) to \(H^1(F)\cap H^{1/2}_{00}(F)\).
    Hence, Corollary~2.4 of \cite{Nic.1992} yields \(\restrict[]{\phi}{F}  \in H^1_0(F)\).

    Since \(-\sigma \vvE_\parallel = \jmp{\vvH_\tau}\) by the Maxwell domain \eqref{eq:MaxwellOperatorDomain}, the trace theorem in \(\vvH(\curl)\) shows \(-\sigma \div_F \vvE_\parallel = \div_F \jmp{\vvH_\tau}\in H^{-1/2}_{00}(F)\), see \cite[Thm.~3.15]{Buf.C.2001a}.
    As a result, \(\div_F\grad_F \phi\) is contained in \(H^{-1/2}_{00}(F)\).
        The operator
        \begin{equation*}
            -\Delta_F \coloneq -\div_F \grad_F \,:\, L^2(F) \supset D(-\Delta_F) \to L^2(F),
            \qquad D(-\Delta_F) \coloneq H^2(F)\cap H^1_0(F)
        \end{equation*}
        is densely defined, self-adjoint and positive with compact resolvent.
        The space \(H^{-1/2}_{00}(F)\) is the adjoint of \(H^{1/2}_{00}(F)\) with respect to \(L^2(F)\).
        The later space is defined in \eqref{eq:InterpolationSpaces}, and can be written by reiteration as
        \begin{equation*}
            H^{1/2}_{00}(F) 
                = \bigl[L^2(F), H^2(F) \cap H^1_0(F)\bigr]_{1/4,2} 
                = \bigl[L^2(F), D(-\Delta_F)\bigr]_{1/4,2}
        \end{equation*}
        Application of the theory of Banach scales generated by \(-\Delta_F\), see \cite[Thm.~V.1.4.12,V.1.5.4]{Ama.1995}, then shows that the inverse \((-\Delta_F)^{-1}\) maps \(H^{-1/2}_{00}(F)\) onto \(H^{3/2}_{0}(F)\) as define in \eqref{eq:InterpolationSpaces}.
        Therefore, since \(-\Delta_F \restrict[]{\phi}{F} \in H^{-1/2}_{00}(F)\), we conclude that
    \begin{equation}\label{eq:FirstEllipticShift}
        \restrict{\phi}{F} \in H^{3/2}(F)\cap H^1_0(F),
        \quad
        \|\restrict{\phi}{F}\|_{H^{3/2}(F)}
        \lesssim \|\div_F \jmp{\vvH_\tau}\|_{H^{-1/2}_{00}(F)} + \|\div_F \vvPhi_\parallel\|_{H^{-1/2}_{00}(F)}.
    \end{equation}
    Corollary~\ref{cor:TangentialRegularityLemma} yields \(\grad_F \phi \in \vvH^{1/2}_{\parallel,00}(F)\), and thus
    \begin{equation*}
        -\frac{1}{\sigma}\jmp{\vvH_\tau} = \vvE_{\parallel} = \vvPhi_{\parallel} + \grad_F\phi \in \vvH^{1/2}_{\parallel,00}(F),
    \end{equation*}
    with
    \begin{equation*}
        \|\jmp{\vvH_\tau}\|_{\vvH^{1/2}_{\parallel,00}(F)} \lesssim \|\vvPhi\|_{H^1(Q_1)^3} + \|\restrict{\phi}{F}\|_{H^{3/2}(F)}
        \lesssim
        \bigl\|(\vvE,\vvH)\bigr\|_{\energyspace},
    \end{equation*}
    using also \eqref{eq:FirstDecomposition}, \eqref{eq:FirstEllipticShift}, and \cite[Prop.~3.13, Thm.~3.15]{Buf.C.2001a}

    2).
    It remains to show that the normal jump of the electric field is also regular.
    Let \(\vvU = \restrict{(\vvR_\parallel \grad_F \phi)}{Q_1} \in H^1(Q_1)^3 \cap \vvH_{\partial Q_1\setminus F}(\curl, Q_1)\) with the right inverse defined in Lemma~\ref{lem:RightInverse}.
    Then \(\vvV = \vvU - \grad \phi\) is an element of \(\vvH_0(\curl,Q_1)\).
    We choose another regular decomposition \(\vvLambda\in H^1(Q_1)^3\) and \(\lambda \in H^1_0(Q_1)\) such that \(\vvV = \vvLambda + \grad \lambda\) with
    \begin{equation*}
        \|\vvLambda\|_{H^1(Q_1)^3} + \|\lambda\|_{H^1(Q_1)} 
        \lesssim \|\vvV\|_{\vvH(\curl,Q_1)}
        \lesssim \|\phi\|_{H^{1}(Q_1)} + \|\grad_F \phi \|_{\vvH^{1/2}_{\parallel,00}(F)}
        \lesssim \bigl\|(\vvE,\vvH)\bigr\|_{\energyspace},
    \end{equation*}
    see, e.g., \cite[Prop.~3.7.2]{Ass.C.L.2018.MFCE}.
    Thus, \(\grad(\lambda + \phi) = \vvU - \vvLambda\) belongs to \(H^1(Q_1)^3\), and therefore \(\restrict{\phi}{\partial Q_1} = \restrict{(\lambda + \phi)}{\partial Q_1}\) to \(H^{3/2}(\partial Q_1)\) with the estimate 
    \begin{equation*}
        \|\restrict{\phi}{\partial Q_1}\|_{H^{3/2}(\partial Q_1)}
        \lesssim
        \|\phi + \lambda\|_{H^1(Q_1)} + \|\vvU+\vvLambda\|_{H^1(Q_1)^3} 
        \lesssim
        \bigl\|(\vvE,\vvH)\bigr\|_{\energyspace} + \|\grad_F \phi\|_{\vvH^{1/2}_{\parallel,00}(F)}
        \lesssim
        \bigl\|(\vvE,\vvH)\bigr\|_{\energyspace}.
    \end{equation*}
    Since \(-\Delta \phi  = \div \vvPhi - \div \vvE^{(1)} = \div\vvPhi \in L^2(Q_1)\), elliptic regularity yields \(\phi \in H^2(Q_1)\cap H^1_{\partial Q_1\setminus F}(Q_1)\), and
    \begin{equation*}
        \|\phi\|_{H^2(Q_1)} 
            \lesssim \|\div\Phi\|_{L^2(Q_1)} + \|\restrict{\phi}{\partial Q_1}\|_{H^{3/2}(\partial Q_1)}
            \lesssim \bigl\|(\vvE,\vvH)\bigr\|_{\energyspace}.
    \end{equation*}
    The first component of \(\vvPhi\) is by the choice of the regular decomposition element of \(H^1(Q_1)\) with zero trace at the boundary facets \(\Gamma_2,  \Gamma_3 \subset \partial Q\), normal to the unit vectors \(\vve_2\) and \(\vve_3\) respectively.
    Lemma~7.3 in \cite{Zer.2021} then shows that the normal trace \(\vvPhi_\perp\) belongs to  \(H^{1/2}_{00}(F)\).
    The scalar potential \(\phi\) belongs to \(H^2(Q_1)\cap H^1_{\partial Q_1 \setminus F}(Q_1)\).
    Lemma~2.1 in \cite{Eil.Sch.2018} in turn shows that \(\partial_n \phi \in H^{1/2}_{00}(F)\).
    We conclude that
    \begin{equation*}
        \vvE^{(1)}_\perp = \vvPhi_\perp + \partial_n \phi \in H^{1/2}_{00}(F).
    \end{equation*}
    Repeating the argument in \(Q_2\) ends the proof, where the asserted estimate follows as in part 1).
\end{proof}

We define the spaces 
\begin{align*}
    \widetilde{\bm{\mathcal{H}}} &= \vvH(\pdiv_\epsilon,Q)\times \vvH_0(\div_\mu,Q),\\
    \widetilde{\bm{\mathcal{V}}} &=
    \setc[\big]{(\vvE,\vvH) \in \widetilde{\bm{\mathcal{H}}}}{
        (\vvE,\vvH) \in \vvH_0(\curl,Q) \times \vvH(\curl,Q),\,\,\jmp{\varepsilon\vvE_\perp} \in H^{1/2}_{00}(F)},
\end{align*}
equipped with their graph norms.
By Lemma~3.1 in \cite[]{Dor.Z.2023}, the embedding \(\widetilde{\bm{\mathcal{V}}}\hookrightarrow PH^1(Q)^6\) is continuous.
Building on the hidden regularity provided by Lemma \ref{lem:FirstHiddenRegularity}, we obtain sufficient control of the normal jump. Combined with the improved regularity of the tangential jump, this permits its removal via the appropriate right inverse from Lemma~\ref{lem:RightInverse}. With the aforementioned embedding, we thus obtain the following result.
\begin{thm}\label{thm:FirstRegularEmbedding}
    The domain \(D(M) \subset \mathcal{V}\) embeds continuously into \(PH^1(Q)^6\).
\end{thm}
\begin{proof}
    Let \(\bigl(\vvE,\vvH\bigr)\in D(M)\) and \(\vvU = -\vvR_\tau \jmp{\vvH_\tau} \in PH^1(Q)^3 \cap \vvH_0(\div_\mu,Q)\), where \(\vvR_\tau\) is defined in Lemma~\ref{lem:RightInverse}.
    By \eqref{eq:ConnectionToFullSpaces}, we conclude that \(\bigl(\vvE, \vvH+\vvU\bigr) \in \widetilde{\bm{\mathcal{V}}} \hookrightarrow PH^1(Q)^3\).
    Furthermore, Lemmata~\ref{lem:RightInverse} and \ref{lem:FirstHiddenRegularity} imply
    \begin{align*}
        \bigl\|(\vvE,\vvH)\bigr\|_{PH^1(Q)^6}
        &\leq \bigl\|(\vvE,\vvH+\vvU)\bigr\|_{PH^1(Q)^6}
            + \bigl\|\vvU\bigr\|_{PH^1(Q)^3}\\
        &\lesssim \bigl\|(\vvE,\vvH+\vvU)\bigr\|_{\widetilde{\bm{\mathcal{V}}}}
            + \bigl\|\jmp{\vvH_\tau}\bigr\|_{\vvH^{1/2}_{\parallel,00}(F)}\\
        &\lesssim \bigl\|(\vvE,\vvH+\vvU)\bigr\|_{\widetilde{\bm{\mathcal{V}}}}
            + \bigl\|(\vvE,\vvH)\bigr\|_{\energyspace}.
    \end{align*}
    Similar, we estimate
    \begin{align*}
        \bigl\|(\vvE,\vvH+\vvU)\bigr\|_{\widetilde{\bm{\mathcal{V}}}}
            &\leq \bigl\|(\vvE,\vvH+\vvU)\bigr\|_{L^2(Q)^6}
                + \bigl\|\curl\vvE\bigr\|_{L^2(Q)^6} + \bigl\|\pcurl\vvH\bigr\|_{L^2(Q)^6}\\
                &\quad\,\,\,+ 
                \bigl\|\pcurl \vvU\bigr\|_{L^2(Q)^3}
                + \bigl\|\div \mu \vvU\|_{L^2(Q)^3}
                + \bigl\|\jmp{\epsilon\vvE_\perp}\bigr\|_{H^{1/2}_{00}(F)}
                + \bigl\|(\vvE,\vvH)\|_{\energyspace}\\
            &\lesssim \bigl\|(\vvE,\vvH)\bigr\|_{\energyspace} + \bigl\|\vvU\bigr\|_{PH^1(Q)^3}%
            \lesssim \bigl\|(\vvE,\vvH)\|_{\energyspace}.
    \end{align*}
\end{proof}

\section{Stability}\label{sec:Stability}

By Theorem~\ref{thm:Existence}, problem \eqref{eq:PDE} is governed by a contraction semigroup, and Corollary~\ref{cor:OperatorIsDissipative} identifies the tangential component of the electric field on the interface as the only dissipation mechanism.
In this section, we investigate for which initial values this dissipation drives the energy to zero as \(t\to\infty\).
Since the damping acts on the interface \(F\) only, decay cannot be expected for all initial values.
We show the following decomposition result.
\begin{prop}\label{prop:Decomp}
    There exists a closed subspace \(\statespace_{\mathrm{u}}\subset \statespace\) with orthogonal projection \(\mathcal{P}_{\mathrm{u}}\,:\, \statespace \to \statespace_{\mathrm{u}}\) such that
    \begin{equation*}
        \bigl\|e^{t M} \vvu_0\bigr\| \to \bigl\|\mathcal{P}_{\mathrm{u}} \vvu_0\bigr\|,
        \quad t\to \infty,
        \qquad \text{ for all } \vvu_0\in \statespace.
    \end{equation*}
    Moreover, the orthogonal decomposition \(\statespace = \statespace_{\mathrm{u}}\oplus\statespace_{\mathrm{s}}\) into \(\statespace_{\mathrm{u}}\) and \(\statespace_{\mathrm{s}} = \statespace_{\mathrm{u}}^\perp\) is maximal in the sense that
    \begin{enumerate}[label=\normalfont (\alph*)]
        \item\label{prop:Decomp:Stable}
            \(\statespace_{\mathrm{s}} = \setc[\big]{\vvu\in\statespace}{\|e^{tM}\vvu\|_{\statespace}\to 0 \text{ as } t\to\infty}\), and
        \item\label{prop:Decomp:Unitary}
            \(\statespace_{\mathrm{u}} = \setc[\big]{\vvu\in\statespace}{\|e^{tM}\vvu\|_{\statespace} = \|\vvu\|_{\statespace} \text{ for all } t\geq 0}\).
    \end{enumerate}
    In particular, \(\statespace_{\mathrm{s}}\) is the largest subspace of \(\statespace\) on which the semigroup \(\bigl(e^{tM}\bigr)_{t\geq 0}\) is strongly stable, and \(\statespace_{\mathrm{u}}\) is the largest subspace on which it acts isometrically.
\end{prop}
The closed subspace \(\statespace_{\mathrm{u}}\) is spanned by eigenfunctions \(\monosetc{\vvTheta_j^\pm}_{j\in J_\mathrm{u}}\) of \(M\) with purely imaginary eigenvalues that are unobserved by the damping.

In the non-refracting case \(\epsilon^{(1)}\mu^{(1)}=\epsilon^{(2)}\mu^{(2)}\), we can explicitly calculate those eigenfunctions.
Under a non-resonance condition on the interface position \(\zeta\in (-1,1)\), we can then characterize the decaying initial fields by two moment conditions.
\begin{thm}\label{thm:FirstOrderStability}
    Assume that \(\epsilon^{(1)}\mu^{(1)}=\epsilon^{(2)}\mu^{(2)}=\nu>0\).
    Let \(\zeta\in(-1,1)\) with \(\zeta\notin\mathbb{Q}\), and let \(\vvu_0 = \bigl(\vvE_0,\vvH_0\bigr)\in\statespace\).
    Then \(\bigl\|e^{tM}\vvu_0\bigr\|_{\statespace}\to 0\) as \(t\to\infty\) if and only if
    \begin{equation}\label{eq:FirstOrderMomentConditions}
        \int_{-1}^{1}\vvE_0\cdot\vve_1\,\mathrm{d}x_1 = 0
        \quad\text{a.e.\ on } (0,1)^2,
        \qquad
        \partial_2 h_3 - \partial_3 h_2 = 0
        \quad\text{in } \mathcal{D}'\bigl((0,1)^2\bigr),
    \end{equation}
    where \(h_l(x_2,x_3) = \int_{-1}^{1}\mu(x_1)\bigl(\vvH_0\cdot\vve_l\bigr)(x_1,x_2,x_3)\,\mathrm{d}x_1\), for \(l\in\monosetc{2,3}\).
\end{thm}

The moment condition on the electric field admits a simple physical interpretation.
For fixed transverse position \((x_2,x_3)\in(0,1)^2\), the segment \(x_1\mapsto (x_1,x_2,x_3)\) connects the two perfectly conducting faces \(\monosetc{x_1 = -1}\) and \(\monosetc{x_1 = 1}\) of the cavity, and it crosses the interface \(F\) orthogonally.
The quantity \(\int_{-1}^{1}\vvE_0\cdot\vve_1\,\mathrm{d}x_1\) is thus the line integral of the initial electric field along this segment, i.e., the initial \emph{voltage} between the two plates measured at \((x_2,x_3)\).
The first condition in \eqref{eq:FirstOrderMomentConditions} therefore states that the initial state carries no voltage across the cavity, for almost every transverse position.
This is precisely the information that the damping cannot detect: the interface feedback in \eqref{eq:MaxwellOperatorDomain} acts on the tangential component \(\vvE_\parallel\) only, whereas the unobserved modes \(\vvE^{\mathrm{TM}}_{0,m_2,m_3}\) of Lemma~\ref{lem:UnobservedModes} are proportional to \(\epsilon^{-1}(x_1)\vve_1\) at each transverse position, like the field of a plate capacitor containing two dielectric layers.
Such a field passes through \(F\) without ever producing a tangential trace, so its energy is not dissipated.

We prove Proposition~\ref{prop:Decomp} in the remainder of this section and Theorem~\ref{thm:FirstOrderStability} in Section~\ref{sec:CavityEigenmodes}. We briefly outline our strategy.
We adapt the ideas of \cite{Amm.Nic.2015}, developed there for second-order systems with collocated feedback, and derive an observability criterion for the eigenmodes of the undamped problem, working directly with the Maxwell operator \(M\) from \eqref{eq:MaxwellOperator}.
The hidden regularity of Section~\ref{sec:Regularity} yields compactness of the resolvent of \(M\) on all of \(\statespace\), so that the long-time behavior is governed by the eigenvalues of \(M\) on the imaginary axis.
For the corresponding eigenfunctions, the dissipation forces the tangential trace of the electric field to vanish on \(F\).
They then carry no surface charge, and the eigenvalue equation for \(M\) reduces to the eigenvalue problem for \(\epsilon^{-1}\curl\mu^{-1}\curl\) on divergence-free fields.
We realize the latter through an auxiliary self-adjoint operator \(A_{\redstatespace}\) with compact resolvent and split its eigenmodes into \emph{observed} ones and \emph{unobserved} ones, whose tangential component vanishes on \(F\) so that they are invisible to the damping.
The unobserved modes lift to an orthonormal system of eigenvectors of \(M\) and of its adjoint, spanning a subspace on which the evolution is unitary, whereas the Arendt--Batty--Lyubich--V\~u theorem yields strong stability on the orthogonal complement.
The decaying initial values are therefore precisely those orthogonal to the lifted unobserved modes, and computing the unobserved modes explicitly proves Theorem~\ref{thm:FirstOrderStability}.

\subsection*{Spectral properties of the Maxwell operator}

We begin with the spectral properties of the Maxwell operator \(M\) from \eqref{eq:MaxwellOperator}.
The key ingredient of our analysis is the hidden regularity of Section~\ref{sec:Regularity}, which yields compactness of the resolvent of \(M\).

\begin{lem}\label{lem:FirstOrderCompactResolvent}
    The Maxwell operator \(M\) has compact resolvent.
    In particular, \(\operatorname{spec}(M)\) consists of at most countably many isolated eigenvalues with finite algebraic multiplicity.
\end{lem}
\begin{proof}
    By Theorem~\ref{thm:Existence} we have \(1\in\rho(M)\), and \((I-M)^{-1}\,:\,\statespace\to D(M)\) is bounded with respect to the graph norm by the closed graph theorem.
    Theorem~\ref{thm:FirstRegularEmbedding} provides the continuous embedding \(D(M)\hookrightarrow PH^1(Q)^6\) with respect to the graph norm.
    The Rellich--Kondrachov theorem on \(Q_i\) for \(i=\monosetc{1,2}\) shows that \(PH^1(Q)^6\hookrightarrow L^2(Q)^6\) is compact.
    Furthermore \(\statespace\) is a closed subspace of \(L^2(Q)^6\).
    Hence, \((I-M)^{-1}\) is compact on \(\statespace\).
    
    Discreteness of the spectrum is a standard consequence of the compact resolvent, see, e.g.,  \cite[Cor.~IV.1.19]{Eng.Nag.2000}.
\end{proof}

The next lemma shows that the first-order system possesses no equilibria.
\begin{lem}\label{lem:FirstOrderKernel}
    It holds that \(\kernel(M) = \monosetc{0}\), 
    hence \(0\in \rho(M)\)
\end{lem}
\begin{proof}
    Let \(\bigl(\vvE,\vvH\bigr)\in \kernel(M)\), i.e., \(\curl\vvE = 0\) and \(\pcurl\vvH = 0\).
    By the dissipativity in Corollary~\ref{cor:OperatorIsDissipative}, we conclude that
    \begin{equation*}
        \sigma \|\vvE_\parallel\|_{L^2(F)^2}^2
            = -\Re \Bigl(M\bigl(\vvE,\vvH\bigr), \bigl(\vvE,\vvH\bigr)\Bigr)_{\statespace}
            = 0.
    \end{equation*}
    Corollary~\ref{cor:ETraceinequalities} then yields
    \[
        \|\jmp{\epsilon\vvE_\perp}\|_{L^2(F)}
        \lesssim \|\vvE_\parallel\|_{L^2(F)^2}
            + \|\curl \vvE\|_{L^2(Q)^3} = 0.
    \]
    The field \(\vvE\) thus belongs to 
    \[
        \vvH_0(\curl = 0,Q) \cap \vvH(\div_\epsilon = 0, Q) = \monosetc{0},
    \]
    where the latter space is known to reduce to zero.
    Moreover, as \(\jmp{\vvH_\tau} = -\sigma\vvE_\parallel = 0\), we similarly obtain
    \[
        \vvH \in \vvH(\curl = 0,Q) \cap \vvH_0(\div_\mu = 0,Q) = \monosetc{0}.
    \]
    Lemma~\ref{lem:FirstOrderCompactResolvent} shows that \(0\in \rho(M)\).
\end{proof}

The following lemma analyzes the eigenfunctions of \(M\) associated with purely imaginary eigenvalues.
Its main point is that these eigenfunctions carry no surface charge, although the state space \(\statespace\) permits charged fields:
the dissipation forces \(\vvE_\parallel = 0\), the interface condition then removes the tangential jump of \(\vvH\), and the eigenvalue equation transfers the resulting global regularity to \(\vvE\).
In particular, the electric field solves an eigenvalue problem for \(\epsilon^{-1}\curl\mu^{-1}\curl\) with homogeneous transmission conditions, which is the starting point for the analysis in the next subsection.

\begin{lem}\label{lem:FirstOrderImaginaryEigenvalues}
    Let \(\beta\in\R\setminus\monosetc{0}\) and \(\bigl(\vvE,\vvH\bigr)\in D(M)\setminus\monosetc{0}\) with \(M\bigl(\vvE,\vvH\bigr) = i\beta\bigl(\vvE,\vvH\bigr)\).
    We then have \(\vvE_\parallel = 0\), \(\jmp{\vvH_\tau} = 0\) on  \(F\), and
    \(\vvE\in \vvH_0(\curl,Q)\cap\vvH(\div_\epsilon = 0, Q)\) with
    \begin{equation}\label{eq:EEigenvalueEquations}
        \epsilon^{-1}\curl\mu^{-1}\curl\vvE = \beta^2\vvE \quad\text{in } L^2(Q)^3,
        \qquad
        \vvH = \tfrac{i}{\beta}\mu^{-1}\curl\vvE.
    \end{equation}
    In particular, \(\vvE \neq 0\) in  \(Q\)  and \(\jmp{\epsilon\vvE_\perp} = 0\) on  \(F\), i.e., the eigenfunction is surface charge-free.
\end{lem}
\begin{proof}
    The eigenvalue equations read as
    \begin{equation}\label{eq:FirstOrderEigenvalueEquations}
        \epsilon^{-1}\pcurl\vvH = i\beta\vvE,
        \qquad
        -\mu^{-1}\curl\vvE = i\beta\vvH.
    \end{equation}
    Taking real parts in Corollary~\ref{cor:OperatorIsDissipative}, we deduce
    \[
        -\sigma\|\vvE_\parallel\|_{L^2(F)^2}^2
            = \Re \bigl(M(\vvE,\vvH),(\vvE,\vvH)\bigr)_{\statespace}
            = \Re i\beta \|(\vvE,\vvH)\|_{L^2(Q)^3}^2 = 0,
    \]
    i.e. \(\vvE_\parallel = 0\) in \(L^2(F)^2\).
    The interface condition in \eqref{eq:MaxwellOperatorDomain} then gives \(\jmp{\vvH_\tau} = 0\), and \eqref{eq:ConnectionToFullSpaces} shows \(\vvH\in\vvH(\curl,Q)\) with \(\pcurl\vvH = \curl\vvH\).

    The first equation in \eqref{eq:FirstOrderEigenvalueEquations} yields \(\vvE = (i\beta)^{-1}\epsilon^{-1}\curl\vvH\), so that \(\div\epsilon\vvE = 0\) holds distributionally in all of \(Q\).
    In particular, \(\jmp{\epsilon\vvE_\perp} = 0\) by \eqref{eq:ConnectionToFullSpaces} and \(\vvE\in\vvH_0(\curl,Q)\cap\vvH(\div_\epsilon = 0, Q)\).
    Equation \eqref{eq:FirstOrderEigenvalueEquations} implies \(\curl\vvE = -i\beta \mu \vvH\) and then
    \begin{equation*}
        \epsilon^{-1}\curl\mu^{-1}\curl\vvE = -i\beta\epsilon^{-1}\curl\vvH = \beta^2\vvE
        \qquad\text{in } L^2(Q)^3.
    \end{equation*}
    Finally, \(\vvE\neq 0\), since \(\vvE = 0\) would imply \(\vvH = \tfrac{i}{\mu\beta}\curl\vvE= 0\).
\end{proof}

\subsection*{The cavity eigenvalue problem}

By Lemma~\ref{lem:FirstOrderImaginaryEigenvalues}, the electric field of an eigenfunction of \(M\) associated with a purely imaginary eigenvalue is divergence-free in all of \(Q\), including the interface, and solves the eigenvalue problem \eqref{eq:EEigenvalueEquations}.
This motivates the study of this eigenvalue problem on the spaces
\begin{equation}\label{eq:ReducedSpaces}
    \redstatespace = \vvH(\div_\epsilon = 0, Q),
    \qquad
    \redenergyspace = \vvH_0(\curl, Q)\cap \vvH(\div_\epsilon = 0, Q),
\end{equation}
where \(\redstatespace\) carries the \(L_\epsilon^2(Q)^3\)-inner product \(\bigl(\epsilon\cdot,\cdot\bigr)_{L^2(Q)^3}\).
On \(\redenergyspace\), the Weber inequality
\begin{equation}\label{eq:StabilityWeber}
    \|\vvV\|_{L^2(Q)^3} \leq C_W \|\curl\vvV\|_{L^2(Q)^3},
    \qquad \vvV\in \redenergyspace,
\end{equation}
is valid with a constant \(C_W = C_W(Q,\epsilon) > 0\), see, e.g., \cite[Thm.~3.4.3]{Ass.C.L.2018.MFCE}.
Therefore, \(\redenergyspace\) is a Hilbert space with the inner product
\begin{equation}\label{eq:VInnerProduct}
    \bigl(\vvU,\vvV\bigr)_{\redenergyspace} = \bigl(\curl \vvU, \curl\vvV\bigr)_{L^2(Q)^3},
    \qquad \vvU,\vvV \in \redenergyspace.
\end{equation}

In contrast to the situation in Section~\ref{sec:WellPosedness}, where the \(L^2(F)\)-regularity of the tangential component had to be encoded into the energy space via \eqref{def:ImprovedRegularityCurlSpaces}, this regularity is automatic on the space \(\redenergyspace\), as the following lemma shows.

\begin{lem}\label{lem:StabilityEmbeddings}
    \begin{enumerate}[label=\normalfont (\alph*)]
        \item\label{lem:StabilityEmbeddings:H1}
            The embedding \(\redenergyspace\hookrightarrow PH^1(Q)^3\) is continuous with \(\|\vvV\|_{PH^1(Q)^3} \lesssim \|\vvV\|_{\redenergyspace}\).
            In particular, the tangential component trace
            \begin{equation*}
                \redenergyspace \to L^2(F)^2,
                \qquad
                \vvV \mapsto \vvV_\parallel,
            \end{equation*}
            is well-defined, bounded, and compact.
        \item\label{lem:StabilityEmbeddings:Compact}
            The embedding \(\redenergyspace\hookrightarrow \redstatespace\) is dense and compact.
    \end{enumerate}
\end{lem}
\begin{proof}
    We prove \ref{lem:StabilityEmbeddings:H1}.
    Proposition~9.3 in \cite{Zer.2021} shows the continuous embedding \(\redenergyspace \hookrightarrow PH^1(Q)^3\).
    The trace theorem then provides boundedness of \(\vvV \mapsto \vvV_\parallel \in H^{1/2}(F)^2\), and the Rellich--Kondrachov theorem shows compactness of \(H^{1/2}(F)^2\hookrightarrow L^2(F)^2\).

    For part \ref{lem:StabilityEmbeddings:Compact}, the compactness follows from \ref{lem:StabilityEmbeddings:H1} and the Rellich--Kondrachov theorem, since \(\redstatespace\) is a closed subspace of \(L^2(Q)^3\).
    It remains to show density.
    Note that \(\redstatespace\) is the orthogonal complement of \(\grad H^1_0(Q)\) in \(L_\epsilon^2(Q)^3\), and denote the corresponding orthogonal projection by \(\vvP\,:\, L_\epsilon^2(Q)^3\to \redstatespace\), cf. \cite[Prop.~6.1.12]{Ass.C.L.2018.MFCE}.
    For \(\vvU\in C^\infty_c(Q)^3\), it holds that \(\vvP\vvU = \vvU - \grad\chi\), where \(\chi\in H^1_0(Q)\) is the unique weak solution of
    \begin{equation}\label{eq:StabilityProjection}
        \bigl(\grad \chi, \grad v\bigr)_{L_\epsilon^2(Q)^3} = \bigl(\vvU,\grad v\bigr)_{L_\epsilon^2(Q)^3},
        \qquad \text{for all } v\in H^1_0(Q).
    \end{equation}
    Indeed, \eqref{eq:StabilityProjection} precisely states that \(\div\epsilon(\vvU - \grad\chi) = 0\) distributionally in \(Q\).
    Since \(\grad H^1_0(Q) \subset \vvH_0(\curl, Q)\cap\kernel(\curl)\), we conclude that \(\vvP \vvU\) belongs to \(\redenergyspace\).
    The density of \(C^\infty_c(Q)^3\) in \(L^2(Q)^3\) and the continuity of \(\vvP\) then imply that \(\vvP C^\infty_c(Q)^3 \subset \redenergyspace\) is dense in \(\redstatespace\).
\end{proof}

On \(\redenergyspace\), we consider the sesquilinear form
\begin{equation}\label{eq:StabilityForm}
    a\,:\, \redenergyspace\times \redenergyspace \to \C,
    \qquad
    a\bigl(\vvU,\vvV\bigr) = \bigl(\mu^{-1}\curl \vvU,\curl\vvV\bigr)_{L^2(Q)^3}, 
\end{equation}
which is densely defined in \(\redstatespace\) by Lemma~\ref{lem:StabilityEmbeddings}~\ref{lem:StabilityEmbeddings:Compact}, symmetric, closed, and coercive by the Weber inequality \eqref{eq:StabilityWeber}.
The associated operator
\begin{equation}\label{eq:AH}
    D(A_{\redstatespace}) = \setc[\big]{\vvU\in \redenergyspace}{\exists\, \vvF\in \redstatespace \text{ with } a\bigl(\vvU,\vvV\bigr) = \bigl(\vvF,\vvV\bigr)_{\redstatespace} \text{ for all } \vvV\in \redenergyspace},
    \qquad A_{\redstatespace}\vvU = \vvF,
\end{equation}
is therefore self-adjoint and positive definite.
Moreover, \(A_{\redstatespace}^{-1}\,:\, \redstatespace\to \redenergyspace\) is bounded by coercivity and the embedding \(\redenergyspace\hookrightarrow \redstatespace\) is compact, so that the operator \(A_{\redstatespace}\) has compact resolvent.

The following lemma identifies the eigenfunctions of \(A_{\redstatespace}\) with the classical Maxwell cavity eigenmodes of \(Q\).
They satisfy homogeneous transmission conditions at the material interface \(F\).
\begin{lem}\label{lem:CavityEigenmodes}
    Let \(\vvPhi\in \redenergyspace\setminus\monosetc{0}\) and \(\beta > 0\).
    Then \(\vvPhi \in D(A_{\redstatespace})\) with \(A_{\redstatespace}\vvPhi = \beta^2\vvPhi\) if and only if
    \begin{equation}\label{eq:CavityEigenproblem}
        \epsilon^{-1}\curl\mu^{-1}\curl\vvPhi = \beta^2 \vvPhi \quad\text{in } L^2(Q)^3.
    \end{equation}
\end{lem}
\begin{proof}
    Let \(A_{\redstatespace} \vvPhi = \beta^2\vvPhi\) and \(\vvU\in C^\infty_c(Q)^3\) be arbitrary.
    With \(\vvP\,:\, L_\epsilon^2(Q)^3\to \redstatespace\) from the proof of Lemma~\ref{lem:StabilityEmbeddings}, the field \(\vvP\vvU = \vvU - \grad\chi\) belongs to \(\redenergyspace\) with \(\chi\in H^1_0(Q)\).
    Hence, we obtain
    \begin{equation*}
        \bigl(\mu^{-1}\curl\vvPhi, \curl \vvU\bigr)_{L^2(Q)^3}
            = a\bigl(\vvPhi, \vvP\vvU\bigr)
            = \beta^2 \bigl(\vvPhi, \vvP\vvU\bigr)_{L_\epsilon^2(Q)^3}
            = \beta^2 \bigl(\epsilon\vvPhi, \vvU\bigr)_{L^2(Q)^3},
    \end{equation*}
    as \(\curl\grad\chi = 0\) and \(\vvP\) is a projection with respect to the \(L_\epsilon^2(Q)^3\) product.
    Equation \eqref{eq:CavityEigenproblem} is thus satisfied in the sense of distributions on all of \(Q\), with right-hand side in \(L^2(Q)^3\).

    Conversely, let \eqref{eq:CavityEigenproblem} hold.
    Then \(\mu^{-1}\curl\vvPhi\in\vvH(\curl,Q)\), and integration by parts shows
    \begin{equation*}
        a\bigl(\vvPhi,\vvU\bigr)
            = \bigl(\curl\mu^{-1}\curl\vvPhi, \vvU\bigr)_{L^2(Q)^3}
            = \beta^2\bigl(\vvPhi,\vvU\bigr)_{L_\epsilon^2(Q)^3},
        \qquad \text{ for all } \vvU \in \redenergyspace \subset \vvH_0(\curl, Q).
    \end{equation*}
    Since \(\beta^2\vvPhi\in \redstatespace\), the definition \eqref{eq:AH} yields \(\vvPhi\in D(A_{\redstatespace})\) with \(A_{\redstatespace}\vvPhi = \beta^2\vvPhi\).
\end{proof}
Note that \eqref{eq:CavityEigenproblem} implies that
\begin{equation}\label{eq:CavityEigenproblemjumpcond}
    \jmp{(\mu^{-1}\curl\vvPhi)_\tau} = 0 \qquad \text{on }  F.
\end{equation}

Since \(A_{\redstatespace}\) is self-adjoint and positive definite with compact resolvent, the spectral theorem provides a countable set of eigenpairs
\begin{equation}\label{eq:EigenpairSet}
    \mathcal{P} = \setc[\big]{\bigl(\beta_j^2, \vvPhi_j\bigr)}{j\in \N},
    \qquad
    A_{\redstatespace} \vvPhi_j = \beta_j^2 \vvPhi_j,\quad \vvPhi_j\neq 0,\quad \beta_j > 0,
\end{equation}
where the eigenvalues are repeated according to their multiplicity and the eigenfunctions form an orthogonal basis of \(\redstatespace\).
The Weber inequality \eqref{eq:StabilityWeber} and the weighted Rayleigh quotient yield the spectral gap \(\beta_j^2 \geq (C_W^2\|\epsilon\|_{L^\infty(Q)}\|\mu\|_{L^\infty(Q)})^{-1}>0\) for all \(j\in \N\).
By Lemma~\ref{lem:StabilityEmbeddings}~\ref{lem:StabilityEmbeddings:H1}, every eigenfunction possesses a tangential trace \(\bigl(\vvPhi_j\bigr)_\parallel\in L^2(F)^2\), through which it is seen by the interface damping.
We choose the eigenbasis adapted to the damping, that is, within each eigenspace we first fix an orthogonal basis of the closed subspace \(\kernel(A_{\redstatespace} - \beta^2)\cap\setc{\vvPhi\in \redenergyspace}{\vvPhi_\parallel = 0}\) and complete it to an orthogonal basis of \(\kernel(A_{\redstatespace}-\beta^2)\) of eigenvalues \(i\beta \in i\R\).
Moreover, we rescale the eigenfunctions such that
\begin{equation}\label{eq:EigenfunctionNormalization}
    \bigl(\vvPhi_j, \vvPhi_i\bigr)_{\redstatespace} = \frac{1}{2\beta_i\beta_j}\,\delta_{ij},
    \qquad i,j\in \N.
\end{equation}
This normalization is chosen such that the lifted eigenfunctions defined in \eqref{eq:LiftedEigenfunctions} below are normalized in \(\statespace\).
By the adapted choice of the basis, we can separate \(\N = J_{\mathrm{o}} \sqcup J_{\mathrm{u}}\) into \emph{observed} and \emph{unobserved} eigenpairs, i.e.,
\begin{equation}\label{eq:ObservedSplit}
    j \in J_{\mathrm{o}} \,\,\Leftrightarrow\,\, \bigl(\vvPhi_j\bigr)_\parallel \neq 0,
    \qquad
    j \in J_{\mathrm{u}} \,\,\Leftrightarrow\,\, \bigl(\vvPhi_j\bigr)_\parallel = 0,
\end{equation}
and it holds that
\begin{equation}\label{eq:AdaptedKernelRepresentation}
    \kernel(A_{\redstatespace} - \beta^2)\cap\setc[\big]{\vvPhi\in \redenergyspace}{\vvPhi_\parallel = 0}
        = \operatorname{span}\setc[\big]{\vvPhi_j}{j\in J_{\mathrm{u}},\, \beta_j = \beta},
    \qquad \beta > 0.
\end{equation}
All eigenpairs of \(\mathcal{P}\) correspond to time-periodic solutions of the undamped Maxwell system.
The observed ones possess a nontrivial tangential trace on \(F\) and are therefore seen by the interface damping, whereas the unobserved ones are invisible to it and give rise to unitary trajectories of the damped semigroup \(\bigl(e^{tM}\bigr)_{t\geq 0}\), as we show in Lemma~\ref{lem:LiftedEigenfunctions} below.
In particular, there is no chance of strong stability on all of \(\statespace\) if \(J_{\mathrm{u}}\neq\emptyset\), and the proof of Proposition~\ref{prop:Decomp} below characterizes the decaying states as those orthogonal to the unobserved modes.

The above results allow us to complete the classification of the imaginary spectrum of \(M\) started in Lemma~\ref{lem:FirstOrderImaginaryEigenvalues}.

\begin{cor}\label{cor:FirstOrderSpectrumOnAxis}
    Let \(\beta\) and \(\bigl(\vvE,\vvH\bigr)\) be as in Lemma~\ref{lem:FirstOrderImaginaryEigenvalues}.
    Then \(\vvE\in D(A_{\redstatespace})\) with \(A_{\redstatespace}\vvE = \beta^2\vvE\), and \(\vvE\) belongs to the span of the unobserved eigenmodes \(\vvPhi_j\) with \(j\in J_{\mathrm{u}}\) and \(\beta_j = |\beta|\).
\end{cor}

\subsection*{The unobserved subspace}

For \(j\in J_{\mathrm{u}}\), we define the lifted unobserved eigenfunctions
\begin{equation}\label{eq:LiftedEigenfunctions}
    \vvTheta_j^\pm = \bigl(\beta_j\vvPhi_j,\, \pm i\mu^{-1}\curl\vvPhi_j\bigr)\in\statespace.
\end{equation}
They arise from the unobserved eigenmodes via the substitution \(\vvH = \tfrac{i}{\beta}\mu^{-1}\curl\vvE\) of Lemma~\ref{lem:FirstOrderImaginaryEigenvalues}, and the normalization \eqref{eq:EigenfunctionNormalization} renders them orthonormal.

\begin{lem}\label{lem:LiftedEigenfunctions}
    \begin{enumerate}[label=\normalfont (\alph*)]
        \item\label{lem:LiftedEigenfunctions:Eigen}
            For \(j\in J_{\mathrm{u}}\), it holds \(\vvTheta_j^\pm\in D(M)\) with \(M\vvTheta_j^\pm = \pm i\beta_j\vvTheta_j^\pm\).
        \item\label{lem:LiftedEigenfunctions:ONS}
            The family \(\setc{\vvTheta_j^\pm}{j\in J_{\mathrm{u}}}\) is an orthonormal system in \(\statespace\).
        \item\label{lem:LiftedEigenfunctions:Adjoint}
            For \(j\in J_{\mathrm{u}}\) we have \(\vvTheta_j^\pm\in D(M^\ast)\) with \(M^\ast\vvTheta_j^\pm = \mp i\beta_j\vvTheta_j^\pm\).
    \end{enumerate}
\end{lem}
\begin{proof}
    We prove \ref{lem:LiftedEigenfunctions:Eigen}.
    Let \(j\in J_{\mathrm{u}}\).
    Lemma~\ref{lem:CavityEigenmodes} yields \(\epsilon^{-1}\curl\mu^{-1}\curl\vvPhi_j = \beta_j^2\vvPhi_j\) in \(L^2(Q)^3\), so that \(\mu^{-1}\curl\vvPhi_j\in\vvH(\curl,Q)\).
    In particular, \(\jmp{(\mu^{-1}\curl\vvPhi_j)_\tau} = 0\) by \eqref{eq:ConnectionToFullSpaces}.
    Moreover, \(\div\curl\vvPhi_j = 0\) in \(Q\), and \(\curl\vvPhi_j\cdot\vvnu = 0\) on \(\partial Q\) since \(\vvPhi_j\times\vvnu = 0\), see, e.g., \cite[Thm.~3.5.1]{Ass.C.L.2018.MFCE}.
    Together with \(\vvPhi_j\in \redenergyspace\) and \(\bigl(\vvPhi_j\bigr)_\parallel = 0\) for \(j\in J_{\mathrm{u}}\), this shows \(\vvTheta_j^\pm\in\energyspace\) with
    \begin{equation*}
        \jmp{\bigl(\pm i\mu^{-1}\curl\vvPhi_j\bigr)_\tau} = 0 = -\sigma \bigl(\beta_j\vvPhi_j\bigr)_\parallel,
    \end{equation*}
    i.e., \(\vvTheta_j^\pm\in D(M)\).
    We compute
    \begin{equation*}
        M\vvTheta_j^\pm
            = \bigl(\pm i\epsilon^{-1}\curl\mu^{-1}\curl\vvPhi_j,\, -\beta_j\mu^{-1}\curl\vvPhi_j\bigr)
            = \bigl(\pm i\beta_j^2 \vvPhi_j,\, -\beta_j\mu^{-1}\curl\vvPhi_j\bigr)
            = \pm i\beta_j \vvTheta_j^\pm.
    \end{equation*}

    For part \ref{lem:LiftedEigenfunctions:ONS}, we use \(\bigl(\mu^{-1}\curl\vvPhi_j,\curl\vvPhi_i\bigr)_{L^2(Q)^3} = a\bigl(\vvPhi_j,\vvPhi_i\bigr) = \beta_j^2\bigl(\vvPhi_j,\vvPhi_i\bigr)_{\redstatespace}\) and compute
    \begin{equation*}
        \bigl(\vvTheta_j^\pm,\vvTheta_i^\pm\bigr)_{\statespace}
            = \bigl(\beta_i\beta_j + \beta_j^2\bigr)\bigl(\vvPhi_j,\vvPhi_i\bigr)_{\redstatespace},
        \qquad
        \bigl(\vvTheta_j^\pm,\vvTheta_i^\mp\bigr)_{\statespace}
            = \bigl(\beta_i\beta_j - \beta_j^2\bigr)\bigl(\vvPhi_j,\vvPhi_i\bigr)_{\redstatespace}.
    \end{equation*}
    For \(i \neq j\), both expressions vanish by the orthogonality in \eqref{eq:EigenfunctionNormalization}, and for \(i = j\) the second expression is zero.
    Finally, the normalization \eqref{eq:EigenfunctionNormalization} yields
    \(
        \|\vvTheta_j^\pm\|_{\statespace}^2
            = 2\beta_j^2 \|\vvPhi_j\|_{\redstatespace}^2
            = 1
    \).

    We prove \ref{lem:LiftedEigenfunctions:Adjoint}.
    Part \ref{lem:LiftedEigenfunctions:Eigen} implies \(e^{tM}\vvTheta_j^\pm = e^{\pm i\beta_j t}\,\vvTheta_j^\pm\) for \(t\geq 0\).
    Using \(\|\vvTheta_j^\pm\|_{\statespace} = 1\) from part \ref{lem:LiftedEigenfunctions:ONS}, we obtain
    \begin{equation*}
        \bigl(e^{tM^\ast}\vvTheta_j^\pm, \vvTheta_j^\pm\bigr)_{\statespace}
            = \bigl(\vvTheta_j^\pm, e^{tM}\vvTheta_j^\pm\bigr)_{\statespace}
            = e^{\mp i\beta_j t}
    \end{equation*}
    and thus, \(\bigl|\bigl(e^{tM^\ast}\vvTheta_j^\pm,\vvTheta_j^\pm\bigr)_{\statespace}\bigr| = 1\).
    Since the adjoint semigroup \(\bigl(e^{tM^\ast}\bigr)_{t\geq 0} = \bigl((e^{tM})^\ast\bigr)_{t\geq 0}\) is also a contraction semigroup, we have \(\|e^{tM^\ast}\vvTheta_j^\pm\|_{\statespace}\leq 1\) and obtain
    \begin{equation*}
        1 = |\bigl(e^{tM^\ast}\vvTheta_j^\pm, \vvTheta_j^\pm\bigr)_{\statespace}|
        \leq \|e^{tM^\ast}\vvTheta_j^\pm\|_{\statespace} \|\vvTheta_j^\pm\|_{\statespace}
        \leq 1.
    \end{equation*}
    Therefore, equality holds in the Cauchy-Schwarz inequality, which implies linear dependence of the two vectors, and we conclude that
    \begin{equation*}
        e^{tM^\ast} \vvTheta_j^\pm = e^{\mp i\beta_j t}\, \vvTheta_j^\pm,
        \qquad t \geq 0,\,\, j\in J_{\mathrm{u}}.
    \end{equation*}
    Differentiation at \(t = 0\) proves the claim.
\end{proof}

Lemma~\ref{lem:LiftedEigenfunctions} shows that the lifted eigenfunctions form an orthonormal system of eigenvectors for both \(M\) and \(M^\ast\).
We define the unitary subspace of unobserved states as
\begin{equation}\label{eq:FirstOrderUnobservedSubspace}
    \statespace_{\mathrm{u}} = \overline{\operatorname{span}}\setc[\big]{\vvTheta_j^\pm}{j\in J_{\mathrm{u}}},
\end{equation}
with the associated orthogonal projection
\begin{equation}\label{eq:FirstOrderUnobservedProjection}
    \mathcal{P}_{\mathrm{u}}\,:\,\statespace\to\statespace_{\mathrm{u}},
    \qquad
    \mathcal{P}_{\mathrm{u}} \vvu
        = \sum_{j\in J_{\mathrm{u}}}
            \Bigl(\bigl(\vvu, \vvTheta_j^+\bigr)_{\statespace}\, \vvTheta_j^+
            + \bigl(\vvu, \vvTheta_j^-\bigr)_{\statespace}\, \vvTheta_j^-\Bigr).
\end{equation}

The eigenfunction relations for \(M\) and \(M^\ast\) imply that \(\mathcal{P}_{\mathrm{u}}\) and \(M\) commute which is needed for our stability analysis.
\begin{lem}\label{lem:FirstOrderUnobservedReducing}
    For every \(\vvu\in D(M)\), it holds that \(\mathcal{P}_{\mathrm{u}}\vvu\in D(M)\) and
    \begin{equation*}
        M\mathcal{P}_{\mathrm{u}}\vvu = \mathcal{P}_{\mathrm{u}}M\vvu .
    \end{equation*}
\end{lem}
\begin{proof}
    Let \(\vvu\in D(M)\).
    By Lemma~\ref{lem:LiftedEigenfunctions}~\ref{lem:LiftedEigenfunctions:Adjoint}, it holds that
    \begin{equation}\label{eq:FirstOrderReducingEigenvalueRelation}
        \bigl(M\vvu, \vvTheta_j^\pm\bigr)_{\statespace}
            = \bigl(\vvu, M^\ast \vvTheta_j^\pm\bigr)_{\statespace}
            = \pm i\beta_j \bigl(\vvu, \vvTheta_j^\pm\bigr)_{\statespace}.
    \end{equation}
    The finite-rank projections
    \begin{equation*}
        \mathcal{P}_{\mathrm{u}}^N \vvu = \sum_{\substack{j\in J_{\mathrm{u}}\\ j\leq N}}
            \bigl(\vvu,\vvTheta_j^+\bigr)_{\statespace}\,\vvTheta_j^+
            + \bigl(\vvu,\vvTheta_j^-\bigr)_{\statespace}\,\vvTheta_j^-
    \end{equation*}
    satisfy \(\mathcal{P}_{\mathrm{u}}^N\vvu\in D(M)\) by Lemma~\ref{lem:LiftedEigenfunctions}~\ref{lem:LiftedEigenfunctions:Eigen} for \(N\in \N\), and \(\mathcal{P}_{\mathrm{u}}^N\vvu\to \mathcal{P}_{\mathrm{u}}\vvu\) in \(\statespace\) as \(N\to\infty\).
    Equation \eqref{eq:FirstOrderReducingEigenvalueRelation} and Lemma~\ref{lem:LiftedEigenfunctions} yield
    \begin{align*}
        M \mathcal{P}_{\mathrm{u}}^N \vvu 
            &= \sum_{\substack{j\in J_{\mathrm{u}}\\ j\leq N}}
            \bigl(\vvu,-i\beta_j\vvTheta_j^+\bigr)_{\statespace}\,\vvTheta_j^+
            + \bigl(\vvu,i\beta_j\vvTheta_j^-\bigr)_{\statespace}\,\vvTheta_j^-\\
            &= \sum_{\substack{j\in J_{\mathrm{u}}\\ j\leq N}}
            \bigl(M\vvu,\vvTheta_j^+\bigr)_{\statespace}\,\vvTheta_j^+
            + \bigl(M\vvu,\vvTheta_j^-\bigr)_{\statespace}\,\vvTheta_j^-
            = \mathcal{P}_{\mathrm{u}}^N M\vvu.
    \end{align*}
    As \(N\to \infty\), the closedness of \(M\) leads to the assertion.
\end{proof}

\subsection*{Strong stability}

The following lemma splits the evolution into a unitary part on the unobserved subspace and a strongly stable part on its orthogonal complement.
The proof of the stability assertion is based on the Arendt--Batty--Lyubich--V\~u theorem, see \cite[Cor.~V.2.22]{Eng.Nag.2000}.

\begin{lem}\label{lem:SemigroupDecomposition}
    The subspaces \(\statespace_{\mathrm{u}}\) and \(\statespace_{\mathrm{s}} = (I - \mathcal{P}_{\mathrm{u}})\statespace = \statespace_{\mathrm{u}}^{\perp}\) are reducing for the semigroup \(\bigl(e^{tM}\bigr)_{t\geq 0}\).
    For every \(\vvu\in\statespace\) it holds that
    \begin{enumerate}[label=\normalfont (\alph*)]
        \item\label{lem:SemigroupDecomposition:Unitary}
            \(\bigl\|e^{tM}\mathcal{P}_{\mathrm{u}}\vvu\bigr\|_{\statespace} = \bigl\|\mathcal{P}_{\mathrm{u}}\vvu\bigr\|_{\statespace}\) for all \(t\geq 0\), and
        \item\label{lem:SemigroupDecomposition:Stable}
            \(\bigl\|e^{tM}(I-\mathcal{P}_{\mathrm{u}})\vvu\bigr\|_{\statespace} \to 0\) as \(t\to\infty\).
    \end{enumerate}
\end{lem}
\begin{proof}
    By Lemma~\ref{lem:FirstOrderUnobservedReducing} and the standard equivalence between commuting orthogonal projections for a generator and for its semigroup, the subspace \(\statespace_{\mathrm{u}}\) is reducing for the semigroup \(\bigl(e^{tM}\bigr)_{t\geq 0}\), and so is the orthogonal complement \(\statespace_{\mathrm{s}}\).
    This equivalence follows, for instance, from the resolvent identity and the exponential formula for \(C^0\)-semigroups, see \cite[Cor.~III.5.5]{Eng.Nag.2000}.

    We prove \ref{lem:SemigroupDecomposition:Unitary}.
    Lemma~\ref{lem:LiftedEigenfunctions} yields
    \begin{align*}
        \bigl\|e^{tM} \mathcal{P}_{\mathrm{u}}\vvu\bigr\|_{\statespace}^2
            = \bigl\|\mathcal{P}_{\mathrm{u}} e^{tM}\vvu\bigr\|_{\statespace}^2
            &= \sum_{j\in J_{\mathrm{u}}}
                \bigl|\bigl(e^{tM}\vvu, \vvTheta_j^+\bigr)_{\statespace}\bigr|^2
                + \bigl|\bigl(e^{tM}\vvu, \vvTheta_j^-\bigr)_{\statespace}\bigr|^2\\
            &= \sum_{j\in J_{\mathrm{u}}}
                \bigl|\bigl(\vvu, e^{tM^\ast}\vvTheta_j^+\bigr)_{\statespace}\bigr|^2
                + \bigl|\bigl(\vvu, e^{tM^\ast}\vvTheta_j^-\bigr)_{\statespace}\bigr|^2
            = \bigl\|\mathcal{P}_{\mathrm{u}}\vvu\bigr\|_{\statespace}^2,
    \end{align*}
    where we used \(e^{tM^\ast}\vvTheta_j^\pm = e^{\mp i\beta_j t}\vvTheta_j^\pm\) from the proof of Lemma~\ref{lem:LiftedEigenfunctions} in the last step.

    We prove \ref{lem:SemigroupDecomposition:Stable}.
    We denote the restriction of the semigroup to \(\statespace_{\mathrm{s}}\) by \(\bigl(e^{tM_{\mathrm{s}}}\bigr)_{t\geq 0}\), with generator \(M_{\mathrm{s}}\) and domain \(D(M_{\mathrm{s}}) = (I-\mathcal{P}_{\mathrm{u}})D(M)\).
    Since \(\mathcal{P}_{\mathrm{u}}\) commutes with the resolvent of \(M\), the restriction \(M_{\mathrm{s}}\) inherits the compact resolvent from Lemma~\ref{lem:FirstOrderCompactResolvent}, so that \(\operatorname{spec}(M_{\mathrm{s}}) = \operatorname{spec}_{\mathrm{p}}(M_{\mathrm{s}})\) is countable.

    Moreover, every eigenfunction of \(M\) with purely imaginary eigenvalue belongs to \(\statespace_{\mathrm{u}}\).
    Indeed, let \(M\bigl(\vvE,\vvH\bigr) = i\beta\bigl(\vvE,\vvH\bigr)\) with \(\bigl(\vvE,\vvH\bigr)\neq 0\), where \(\beta\neq 0\) by Lemma~\ref{lem:FirstOrderKernel}.
    Corollary~\ref{cor:FirstOrderSpectrumOnAxis} provides the representation \(\vvE = \sum_{j\in J_{\mathrm{u}}} c_j\vvPhi_j\), and Lemma~\ref{lem:FirstOrderImaginaryEigenvalues} gives \(\vvH = \tfrac{i}{\beta}\mu^{-1}\curl\vvE\), so that
    \begin{equation*}
        \bigl(\vvE,\vvH\bigr)
            = \sum_{\substack{j\in J_{\mathrm{u}}\\ \beta_j = |\beta|}} \frac{c_j}{\beta_j}\, \vvTheta_j^{\pm}
            \in \statespace_{\mathrm{u}}.
    \end{equation*}
    Therefore, \(\kernel(i\beta - M)\cap \statespace_{\mathrm{s}} \subset \statespace_{\mathrm{u}}\cap\statespace_{\mathrm{s}} = \monosetc{0}\), i.e., \(\operatorname{spec}_{\mathrm{p}}(M_{\mathrm{s}})\cap i\R = \emptyset\).
    So the Arendt--Batty--Lyubich--V\~u theorem shows that \(\bigl(e^{tM_{\mathrm{s}}}\bigr)_{t\geq 0}\) is strongly stable, see \cite[Cor.~V.2.22]{Eng.Nag.2000}.
    Since \((I-\mathcal{P}_{\mathrm{u}})\vvu \in \statespace_{\mathrm{s}}\), assertion \ref{lem:SemigroupDecomposition:Stable} follows.
\end{proof}

We can now prove the decomposition result announced at the beginning of the section.

\begin{proof}[Proof of Proposition~\ref{prop:Decomp}]
    We choose the closed subspace \(\statespace_{\mathrm{u}}\) from \eqref{eq:FirstOrderUnobservedSubspace} with the orthogonal projection \(\mathcal{P}_{\mathrm{u}}\) from \eqref{eq:FirstOrderUnobservedProjection}, and put \(\statespace_{\mathrm{s}} = \statespace_{\mathrm{u}}^\perp\).
    Since \(\statespace_{\mathrm{u}}\) and \(\statespace_{\mathrm{s}}\) are reducing by Lemma~\ref{lem:SemigroupDecomposition}, the trajectory splits orthogonally as \(e^{tM}\vvu = e^{tM}\mathcal{P}_{\mathrm{u}}\vvu + e^{tM}(I-\mathcal{P}_{\mathrm{u}})\vvu\), and parts \ref{lem:SemigroupDecomposition:Unitary} and \ref{lem:SemigroupDecomposition:Stable} of the lemma yield
    \begin{equation}\label{eq:FirstOrderTrajectorySplit}
        \bigl\|e^{tM}\vvu\bigr\|_{\statespace}^2
            = \bigl\|\mathcal{P}_{\mathrm{u}}\vvu\bigr\|_{\statespace}^2
                + \bigl\|e^{tM}(I-\mathcal{P}_{\mathrm{u}})\vvu\bigr\|_{\statespace}^2
            \to \bigl\|\mathcal{P}_{\mathrm{u}}\vvu\bigr\|_{\statespace}^2,
        \qquad t\to\infty,
    \end{equation}
    for every \(\vvu\in\statespace\).
    This is the asserted convergence.

    We prove \ref{prop:Decomp:Stable}.
    By \eqref{eq:FirstOrderTrajectorySplit}, the trajectory tends to \(0\) if and only if \(\mathcal{P}_{\mathrm{u}}\vvu = 0\), so that
    \begin{equation}\label{eq:FirstOrderDecaySet}
        \setc[\big]{\vvu\in\statespace}{\|e^{tM}\vvu\|_{\statespace}\to 0 \text{ as } t\to\infty}
            = \kernel\mathcal{P}_{\mathrm{u}}
            = \statespace_{\mathrm{s}}.
    \end{equation}
    In particular, every subspace of \(\statespace\) on which the semigroup is strongly stable is contained in \(\statespace_{\mathrm{s}}\), and \(\statespace_{\mathrm{s}}\) itself is strongly stable by Lemma~\ref{lem:SemigroupDecomposition}~\ref{lem:SemigroupDecomposition:Stable}.

    We prove \ref{prop:Decomp:Unitary}.
    Let \(\vvu\in\statespace\) with \(\|e^{tM}\vvu\|_{\statespace} = \|\vvu\|_{\statespace}\) for all \(t\geq 0\).
    Since \(\|\vvu\|_{\statespace}^2 = \|\mathcal{P}_{\mathrm{u}}\vvu\|_{\statespace}^2 + \|(I-\mathcal{P}_{\mathrm{u}})\vvu\|_{\statespace}^2\), the identity in \eqref{eq:FirstOrderTrajectorySplit} implies
    \begin{equation*}
        \bigl\|e^{tM}(I-\mathcal{P}_{\mathrm{u}})\vvu\bigr\|_{\statespace}
            = \bigl\|(I-\mathcal{P}_{\mathrm{u}})\vvu\bigr\|_{\statespace},
        \qquad t\geq 0.
    \end{equation*}
    Letting \(t\to\infty\), Lemma~\ref{lem:SemigroupDecomposition}~\ref{lem:SemigroupDecomposition:Stable} yields \((I-\mathcal{P}_{\mathrm{u}})\vvu = 0\), i.e., \(\vvu\in\statespace_{\mathrm{u}}\).
    Conversely, every \(\vvu\in\statespace_{\mathrm{u}}\) satisfies \(\|e^{tM}\vvu\|_{\statespace} = \|\vvu\|_{\statespace}\) for all \(t\geq0\) by Lemma~\ref{lem:SemigroupDecomposition}~\ref{lem:SemigroupDecomposition:Unitary}.
    Hence, any subspace of \(\statespace\) on which the semigroup acts isometrically is contained in \(\statespace_{\mathrm{u}}\).
\end{proof}

\section{Eigenmodes of the cavity operator}\label{sec:CavityEigenmodes}

Following the arguments from \cite{Cos.Dau.2019}, we characterize the eigenmodes of the cavity operator $A_{\redstatespace}$ from \eqref{eq:AH} in the non-refracting case
\begin{equation}\label{partcasecoef}
\epsilon^{(1)}\mu^{(1)}= \epsilon^{(2)}\mu^{(2)}.
\end{equation}

 Given Snell's law \cite[§7.3]{Jac.1975.CE}, this means that no refraction at the interface occurs. 
 For brevity, under the assumption \eqref{partcasecoef}, we set
 \[
 \nu=\epsilon^{(1)}\mu^{(1)}= \epsilon^{(2)}\mu^{(2)}.
 \]
We then identify the unobserved modes and prove Theorem~\ref{thm:FirstOrderStability}.
 
 \subsection{Eigenmodes of one-dimensional operators}
First, given $\zeta\in (-1,1)$, let us introduce the operator
 $B_{\rm Dir}$ in $L^2_\epsilon (-1,1)$ defined by 
 \[
 D(B_{\rm Dir})
  =\{w\in H^1_0(-1,1)\cap PH^2(-1,1)\, |\, \jmp{\mu^{-1} \partial_1 w}(\zeta)=0\},
  \quad B_{\rm Dir} w= -\epsilon^{-1} \partial_1(\mu^{-1}\partial_1 w),
 \]
 where
 $\jmp{v}(\zeta):=v(\zeta+)-v(\zeta-)$, and $PH^2(-1,1)$ is the set of functions in $L^2(-1,1)$ whose restrictions belong to $H^2(-1,\zeta)$ respectively $H^2(\zeta, 1)$.
 We denote with $\partial_1$ the partial derivative with respect to the variable $x_1$.
 This operator is self-adjoint with compact resolvent in $L^2_\epsilon (-1,1)$ which is equipped with the inner product
 \[
 (w_1,w_2)_\epsilon=\int_{-1}^1 w_1 \overline{w_2} \epsilon \,\mathrm{d}x_1,\qquad \forall w_1, w_2\in L^2 (-1,1).
 \]
Indeed, it is associated with the continuous, sesquilinear and coercive form 
 \[
 a(w_1,w_2) =\int_{-1}^1 w'_1 \overline{w'_2} \mu^{-1} \,\mathrm{d}x_1,\qquad  \forall w_1, w_2\in H^1_0(-1,1).
 \]
Hence, the spectrum of $B_{\rm Dir}$ is discrete, and the associated eigenvectors form an orthonormal basis of $L^2_\epsilon (-1,1)$.
This set of eigenmodes is characterized in the next result.
\begin{lem}\label{lem:BDir}
The spectrum of $B_{\rm Dir}$ is the set of positive real numbers $\lambda^2$ (with $\lambda>0$) that satisfy
\begin{equation}\label{characteristic_eq}
\mu^{(1)} \cos(\sqrt{\nu} \lambda (\zeta-1)) \sin(\sqrt{\nu} \lambda(\zeta+1))
-\mu^{(2)}\cos(\sqrt{\nu} \lambda (\zeta+1)) \sin(\sqrt{\nu} \lambda (\zeta-1))=0.
\end{equation}
Each eigenvalue is simple and its associated eigenvector is given by
\begin{equation}\label{splitting_eigenvector}
w(x_1)=
\left\{
\begin{array}{ll}
a_1 \sin (\sqrt{\nu} \lambda (x_1+1)) &\text{ if } x_1\leq \zeta,
\\
a_2 \sin (\sqrt{\nu} \lambda (x_1-1)) &\text{ if } x_1\geq \zeta,
\end{array}
\right.
\end{equation}
where the nonzero real coefficients $a_1$ and $a_2$ are determined up to a common scalar factor by the transmission conditions.
\end{lem}
\begin{proof}
Let $\lambda^2$ be an eigenvalue of  $B_{\rm Dir}$ (with $\lambda>0$).
Then its associated eigenvector  $w\in D(B_{\rm Dir})$ solves
\[
-\partial_1(\mu^{-1}\partial_1 w)=\epsilon \lambda^2 w \qquad \text{ in } (-1,1).
\]
Due to the Dirichlet boundary condition at $-1$ and 1, $w$ is given by \eqref{splitting_eigenvector} for some real numbers $a_1$ and $a_2$. So it remains the transmission conditions at $\zeta$. First the continuity condition of $w$ is equivalent to
\[
a_1 \sin (\sqrt{\nu} \lambda (\zeta+1))=
a_2 \sin (\sqrt{\nu} \lambda (\zeta-1)).
\]
Second the jump condition $\jmp{\mu^{-1} \partial_1 w}=0$ here reduces to
\[
a_1 (\mu^{(1)})^{-1}\cos (\sqrt{\nu} \lambda (\zeta+1))=
a_2 (\mu^{(2)})^{-1}\cos (\sqrt{\nu} \lambda (\zeta-1)),
\]
since $\sqrt{\nu}\lambda$ is positive. This means that a nontrivial solution exists if and only if
\[
\det \left(
\begin{array}{ll}
\sin (\sqrt{\nu} \lambda (\zeta+1)) & -\sin (\sqrt{\nu} \lambda (\zeta-1))\\
(\mu^{(1)})^{-1}\cos (\sqrt{\nu} \lambda (\zeta+1)) &-(\mu^{(2)})^{-1}\cos (\sqrt{\nu }\lambda (\zeta-1))
\end{array}
\right)=0,
\]
which is equivalent to \eqref{characteristic_eq}. Finally such an eigenvalue is simple because
$\sin (\sqrt{\nu} \lambda (\zeta+1))$ and $\cos (\sqrt{\nu}\lambda (\zeta+1))$ cannot both be equal to zero.
\end{proof}

Similarly, we introduce the operator $B_{\rm Neu}$ in $L^2_\mu (-1,1)$ defined by
 \begin{align*}
  D(B_{\rm Neu})
    &=\{w\in H^1(-1,1)\cap PH^2(-1,1)\, |\, \jmp{\epsilon^{-1} \partial_1 w}(\zeta)=0
  \text{ and } w'(-1)=w'(1)=0\},\qquad 
  \\B_{\rm Neu} w&= -\mu^{-1}\partial_1(\epsilon^{-1} \partial_1 w).
 \end{align*}
 Again this operator is self-adjoint with compact resolvent in $L^2_\mu (-1,1)$.
As before, we can determine the set of eigenmodes.
\begin{lem}\label{lem:BNeu}
Assume \eqref{partcasecoef}.
First, \(0\) is an eigenvalue of $B_{\rm Neu}$ with constant eigenvector.
Second  a positive real number    $\lambda^2$ (with $\lambda>0$) is an eigenvalue of  $B_{\rm Neu}$ if and only if 
\eqref{characteristic_eq} holds.
Each positive eigenvalue is simple and its associated eigenvector is given by
\begin{equation}\label{splitting_eigenvector_Neumann}
w(x_1)=
\left\{
\begin{array}{ll}
a_1 \cos (\sqrt{\nu} \lambda (x_1+1)) &\text{ if } x_1\leq \zeta,
\\
a_2 \cos(\sqrt{\nu} \lambda (x_1-1)) &\text{ if } x_1\geq \zeta,
\end{array}
\right.
\end{equation}
where the nonzero real coefficients $a_1$ and $a_2$ are determined up to a common scalar factor by the transmission conditions.
\end{lem}
\begin{proof}
The quadratic form of $B_{\rm Neu}$ shows that its kernel consists of the constant functions.
For $\lambda>0$, the equation $B_{\rm Neu}w=\lambda^2w$ and the boundary conditions imply \eqref{splitting_eigenvector_Neumann}.
Continuity of $w$ and $\epsilon^{-1}w'$ at $\zeta$ yields a nontrivial solution precisely when
\[
\epsilon^{(1)}\cos\bigl(\sqrt{\nu}\lambda(\zeta+1)\bigr)
    \sin\bigl(\sqrt{\nu}\lambda(\zeta-1)\bigr)
-\epsilon^{(2)}\cos\bigl(\sqrt{\nu}\lambda(\zeta-1)\bigr)
    \sin\bigl(\sqrt{\nu}\lambda(\zeta+1)\bigr)=0.
\]
Substituting $\epsilon^{(i)}=\nu/\mu^{(i)}$ gives \eqref{characteristic_eq}.
The proof can be complemented as in the previous lemma.
\end{proof}

Assumption \eqref{partcasecoef} gives the same longitudinal wavenumber $\sqrt{\nu}\lambda$ on both sides of the interface, as used in the preceding formulas.
We denote by $(\lambda_n^2)_{n\in \N}$ the common positive eigenvalues of $B_{\rm Dir}$ and \(B_{\rm Neu}\), in increasing order.
The orthonormal set of Dirichlet eigenfunctions are denoted with $(w_n)_{n\in \N}$, and the Neumann eigenfunctions with 
$\tilde w_n$.
The zero eigenvalue of $B_{\rm Neu}$ is denoted with \(\lambda_0 = 0\), and the associated eigenfunction is given by
$\tilde w_0=\left(\int_{-1}^1 \mu\, \mathrm{d}x_1\right)^{-1/2}$.

\subsection{TE and TM modes}

For \(n_1, m_2, m_3\in\N_0\), we define the wavenumbers
\begin{equation}\label{eq:Wavenumbers_general}
    k_1 = \begin{cases}
      \lambda_{n_1}, & n_1 > 0,\\
      0, &n_1 = 0,
    \end{cases}
    \qquad k_2 = m_2\pi,
    \qquad k_3 = m_3\pi,
    \qquad
    \beta_{n_1,m_2,m_3}^2 = k_1^2 + \nu^{-1}(k_2^2 + k_3^2).
\end{equation}
The eigenmodes come in two families, the transverse electric (TE) and the transverse magnetic (TM) modes, both with eigenvalue \(\beta_{n_1,m_2,m_3}^2\).
The TE modes are defined for \(n_1\geq 1\) and \(m_2 + m_3 > 0\) by
\begin{equation}\label{eq:TEModesgeneral}
    \vvE^{\mathrm{TE}}_{n_1,m_2,m_3}(x_1,x_2,x_3)
        \simeq w_{n_1}(x_1)
        \begin{pmatrix}
            0\\[2pt]
            \phantom{-}m_3 \cos(k_2x_2)\sin(k_3x_3)\\
            -m_2\sin(k_2x_2)\cos(k_3x_3)
        \end{pmatrix},
\end{equation}
and the TM modes for \(n_1\geq 0\) and \(m_2, m_3\geq 1\) by
\begin{equation}\label{eq:TMModesgeneral}
    \vvE^{\mathrm{TM}}_{n_1,m_2,m_3}(x_1,x_2,x_3)
        \simeq
        \begin{pmatrix}
            (k_2^2+k_3^2)\epsilon^{-1} \tilde w_{n_1}(x_1)\sin(k_2x_2)\sin(k_3x_3)\\
            k_2 \epsilon^{-1} \widetilde{\partial_1} \tilde w_{n_1}(x_1)\cos(k_2x_2)\sin(k_3x_3)\\
            k_3\epsilon^{-1} \widetilde{\partial_1} \tilde w_{n_1}(x_1)\sin(k_2x_2)\cos(k_3x_3)
        \end{pmatrix}.
\end{equation}
Here, $\simeq$ denotes equality up to a nonzero normalization constant, and
\[
\widetilde{\partial_1}w(x_1)
=
\begin{cases}
    \partial_1 w^{(1)}(x_1), & x_1<\zeta,\\
    \partial_1 w^{(2)}(x_1), & x_1>\zeta,
\end{cases}
\] is the piecewise derivative of $w$.
It can be checked that these functions form an orthogonal set in \(L^2_\epsilon(Q)\), using also \eqref{partcasecoef}.

Indeed, as in \cite{Cos.Dau.2019} for the case \(\epsilon=1=\mu\), we can prove the following results.

\begin{lem}\label{lemmaeigenfctAH}
Under the assumption \eqref{partcasecoef}, the TE-modes   $\vvE^{\mathrm{TE}}_{n_1,m_2,m_3}$, with $n_1\geq 1, m_2 + m_3 > 0$ 
and the TM-modes   $\vvE^{\mathrm{TM}}_{n_1,m_2,m_3}$ with $n_1\geq 0, m_2, m_3\geq 1$
are  eigenvectors of $A_{\redstatespace}$. In both cases, the associated eigenvalue is
$\beta_{n_1,m_2,m_3}^2$.
\end{lem}
\begin{proof}
We first look for TE-modes in the form
\begin{equation}\label{TEmodesCD}
\vvu(x_1, x_\perp)=\left(
\begin{array}{ll}
0\\
\curl_\perp v(x_\perp)
\end{array}
\right) w(x_1),
\end{equation}
where $x=(x_1,x_2, x_3)=(x_1, x_\perp)$ and
\[
\curl_\perp v(x_\perp)=\left(
\begin{array}{ll}
\partial_3 v\\
-\partial_2 v
\end{array}
\right),
\]
see \cite[(2.7)]{Cos.Dau.2019}.
In the following, we also use the differential operators for vector fields \(\vvv = (v_1,v_2)\)
\begin{equation*}
    \scurl_\perp \vvv = \partial_1 v_2 -\partial_2 v_1,
    \qquad
    \div_\perp \vvv = \partial_1 v_1 + \partial_2 v_2.
\end{equation*}

We notice that $\div(\epsilon \vvu)=0$ since $u_1=0$ and \(\div_\perp\curl_\perp = 0\).
From the definition of $\vvu$ the electric boundary conditions on $\partial Q$ become
\begin{equation}\label{electicbcTE}
\partial_n v=0 \quad \text{ on } \partial (0,1)^2,
\qquad w(-1)=w(1)=0.
\end{equation}
Further the jump conditions 
\begin{equation}\label{interfacecond}
\jmp{\vvu_\tau}= 0 = \jmp{\mu^{-1} (\curl\vvu)_\tau} \quad \text{ on } F,
\end{equation}
are equivalent to 
\begin{equation}\label{jumpTE}
\jmp{w}(\zeta) = 0 = \jmp{\mu^{-1} \partial_1 w}(\zeta).
\end{equation}
 As in equation (3.1) of  \cite{Cos.Dau.2019}, we see that
 \begin{equation}\label{eigenpbAH}
   \epsilon^{-1}\curl\mu^{-1}\curl\vvu = \beta^2 \vvu,
   \end{equation}
is satisfied if
\begin{equation}\label{eigenpbv}
-\epsilon^{-1} \mu^{-1} \Delta_\perp v=\lambda_{\perp} v\quad \text{ in } (0,1)^2,
 \end{equation}
and
 \[
-\epsilon^{-1} \partial_1(\mu^{-1} \partial_1 w)=\lambda_\parallel w\quad \text{ in } (-1,1),
\]
with $\lambda_{\perp}+\lambda_\parallel=\beta^2$ and \(\Delta_{\perp} = \partial_2^2 + \partial_3^2\).
Using the boundary condition \eqref{electicbcTE}  on $v$ and $\epsilon \mu=\nu$, we deduce that $v(x_\perp) \simeq \cos(k_2x_2) \cos(k_3x_3)$ and
\begin{equation}\label{eigenTEv}
\lambda_{\perp}=\nu^{-1} (m_2^2+m_3^2)\pi^2\quad 
\text{ for } m_2, m_3\in \mathbb{N}_0.
\end{equation}

The Dirichlet boundary condition \eqref{electicbcTE}  on $w$ and the jump condition  \eqref{jumpTE} imply that $w\in D(B_{\rm Dir})$, \(w\simeq w_n\), and 
\begin{equation}\label{eigenTEvparallel}
\lambda_{\parallel}=\lambda_n^2\quad 
\text{ for } n\in \mathbb{N}.
\end{equation}
This ends the proof for the TE-modes because the case $m_2=m_3=0$  yields  $\vvu=\mathbf{0}$.

We now look for TM-modes in the form
\begin{equation}\label{TMmodesCD}
\vvu(x_1, x_\perp)=\left(
\begin{array}{ll}
0\\
\nabla_\perp v(x_\perp)
\end{array}
\right) \widetilde{\partial_1} w(x_1)-\left(
\begin{array}{ll}
\Delta_\perp v(x_\perp)
\\
0
\end{array}
\right)  w(x_1),
\end{equation}
see \cite[(2.8)]{Cos.Dau.2019}, where 
\[
\nabla_\perp v(x_\perp)=\left(
\begin{array}{ll}
\partial_2 v\\
\partial_3 v
\end{array}
\right).
\]

First, using \eqref{eq:ConnectionToFullSpaces} we see that $\div(\epsilon \vvu)=0$ if and only if  $\jmp{ \epsilon u_1}=0$ on $F$
or equivalently
\begin{equation}\label{jumpTMdiv}
\jmp{\epsilon w}(\zeta)=0.
\end{equation}

The electric boundary conditions on $\partial Q$ can be written as
\begin{equation}\label{electicbcTM}
 v=0\quad \text{ on } \partial (0,1)^2,\qquad  w'(-1)=w'(1)=0.
\end{equation}
Further the jump conditions \eqref{interfacecond} hold if
\begin{equation}\label{jumpTM}
  \jmp{\partial_1 w}(\zeta) = 0,\,\,\, \text{and }\quad  
  \jmp{\partial_1 (\mu^{-1} \partial_1 w)}(\zeta) = 0,\,\,\, \text{and }\quad
  \jmp{\mu^{-1} w}(\zeta)= 0.
\end{equation}
 Finally, from the equation (3.5) of \cite{Cos.Dau.2019}, applied in $Q_i$ each, we see that \(\vvu\) solves
 \eqref{eigenpbAH} if \eqref{eigenpbv} and
\begin{equation}\label{eigenTMw}
-\epsilon^{-1}  \partial_1 (\mu^{-1} \partial_1) w=\lambda_\parallel w\quad  \text{ in } (-1,\zeta)\cup (\zeta, 1),
\end{equation}
with $\lambda_{\perp}+\lambda_\parallel=\beta^2$ hold.
From the boundary condition \eqref{electicbcTM}  on $v$ and \eqref{eigenpbv}, we deduce that  $v(x_\perp) \simeq \sin(k_2x_2) \sin(k_3x_3)$ and
\[
\lambda_{\perp}=\nu^{-1} (m_2^2+m_3^2)\pi^2 \quad
\text{ for } m_2, m_3\in \N.
\]

The situation is more delicate for the eigenvalue problem for $w$. 
First we notice that assumption \eqref{partcasecoef} implies that the condition $\jmp{\mu^{-1} w}(\zeta)= 0$ in \eqref{jumpTM} is equivalent to \eqref{jumpTMdiv}. 
Further the condition 
$\jmp{\partial_1 (\mu^{-1} \partial_1 w)}(\zeta)= 0$ is a consequence of \eqref{jumpTMdiv} and of \eqref{eigenTMw}. 
In summary, the eigenvalue problem for $w$ reads
\begin{equation}\label{eigenpbTMw}
\left\{
\begin{array}{lll}
-\epsilon^{-1} \partial_1(\mu^{-1} \partial_1 w)=\lambda_\parallel w\quad \text{ in } (-1,\zeta)\cup (\zeta, 1),
\\
w'(-1)=w'(1)=0,\\
\jmp{\epsilon w}(\zeta)=0,\quad 
\jmp{\partial_1 w}(\zeta)=0.
\end{array}
\right.
\end{equation}
We now set $\tilde w=\epsilon w$, and notice that the previous system is equivalent to
 \begin{equation}\label{eigenpbTMtildew}
\left\{
\begin{array}{lll}
-\mu^{-1} \partial_1(\epsilon^{-1} \partial_1 \tilde w)=\lambda_\parallel \tilde w\quad \text{ in } (-1,\zeta)\cup (\zeta, 1),
\\
\tilde w'(-1)=\tilde w'(1)=0,\\
\jmp{\tilde w}(\zeta)=0,\quad  
\jmp{\epsilon^{-1}\partial_1 \tilde w}(\zeta)=0.
\end{array}
\right.
\end{equation}
We deduce that $\tilde w\in D(B_{\rm Neu})$ and 
$B_{\rm Neu} \tilde w= \lambda_\parallel  \tilde w$ in $L^2_\mu(-1,1)$. This shows that $w\simeq \epsilon^{-1}\tilde w_n$ with
\[
\lambda_{\parallel}= \lambda_n^2,
\text{ for } n\in \mathbb{N}_{0}.
\]
\end{proof}

\begin{thm}\label{thm:completeness}
Under the assumption \eqref{partcasecoef}, a complete family of eigenvectors of $A_{\redstatespace}$ is given by $\vvE^{\mathrm{TE}}_{n_1,m_2,m_3}$, with $n_1\geq 1,\, m_2 + m_3 > 0$, 
and  $\vvE^{\mathrm{TM}}_{n_1,m_2,m_3}$ with $n_1\geq 0,\, m_2, m_3\geq 1$. 
In both cases, the associated eigenvalue is
$\beta_{n_1,m_2,m_3}^2$. These eigenvectors form an orthogonal basis of $\redstatespace$.
\end{thm}
\begin{proof}
We recall that by Lemma \ref{lem:StabilityEmbeddings} the space \(\redenergyspace\) is dense in  \(\redstatespace\), and that the embedding \(\redenergyspace\hookrightarrow PH^1(Q)^3\) is continuous.

Let us fix a TM-mode $\vvE^{\mathrm{TM}}_{n_1,m_2,m_3}$ for \(n_1\geq 0\) and \(m_2, m_3\geq 1\).
We write $v(x_\perp) =\sin(k_2x_2) \sin(k_3x_3)$, $w(x_1)=\epsilon^{-1} \tilde w_{n_1}$ and choose the real representative in \eqref{eq:TMModesgeneral} without a normalization factor. We set
\[
I_{n_1,m_2,m_3}(\vvu)=\int_Q \vvu(x)\cdot \vvE^{\mathrm{TM}}_{n_1,m_2,m_3}(x)\, \epsilon \mathrm{d}x
\qquad \vvu \in \redstatespace.
\]

In a first step, we take \(\vvu\in \redenergyspace\) and transform $I_{n_1,m_2,m_3}(\vvu)$.
As $\epsilon=\nu \mu^{-1}$, we have
 \[
I_{n_1,m_2,m_3}(\vvu)=\nu  \int_Q \left(\vvu_\perp (x)\cdot 
\nabla_\perp v(x_\perp) 
\widetilde{\partial_1} w(x_1)-u_1(x)
\Delta_\perp v(x_\perp)
   w(x_1) \right) \mu^{-1}\, \,\mathrm{d}x.
  \]
The first term on the right-hand side is integrated by parts in the transverse variables.
Since $v=0$ on $\partial (0,1)^2$, no boundary term occurs, and \eqref{eigenpbv} gives
\[
I_{n_1,m_2,m_3}(\vvu)=\nu\int_Q \left(-\div_\perp\vvu_\perp (x) 
 v(x_\perp) 
\widetilde{\partial_1} w(x_1)+\nu \lambda_\perp u_1(x)
  v(x_\perp)
   w(x_1) \right) \mu^{-1}\, \,\mathrm{d}x.
   \]
   Since $0=\div\vvu=\partial_1 u_1+\div_\perp\vvu_\perp$ on $Q_i, i=1,2$, we obtain
   \[
I_{n_1,m_2,m_3}(\vvu)=\nu\int_Q\left( v(x_\perp)\widetilde{\partial_1}  u_1
\widetilde{\partial_1} w(x_1)+\nu \lambda_\perp u_1(x)
  v(x_\perp)
   w(x_1)\right) \mu^{-1}\, \,\mathrm{d}x.
   \]
We integrate the first term on the right-hand side by parts with respect to $x_1$ on $(-1,\zeta)$ and $(\zeta,1)$.
The endpoint terms vanish since $w'(-1)=w'(1)=0$.
The interface terms cancel because $\jmp{\mu^{-1}u_1}=\nu^{-1}\jmp{\epsilon u_1}=0$ for $\vvu\in\redenergyspace$ and $\jmp{w'}=0$ by \eqref{eigenpbTMw}. Thus,
\[
  I_{n_1,m_2,m_3}(\vvu)=\nu\int_Q\left(-v(x_\perp) u_1(x)
  (\widetilde{\partial_1})^2 w(x_1)+\nu \lambda_\perp u_1(x)
    v(x_\perp)
     w(x_1)\right) \mu^{-1}\, \,\mathrm{d}x.
\]
The first identity of \eqref{eigenpbTMw} and \eqref{partcasecoef} imply
   \[
I_{n_1,m_2,m_3}(\vvu)=\nu (\lambda_\parallel+\lambda_\perp)\int_Q u_1(x)
  v(x_\perp)
   w(x_1) \epsilon\, \,\mathrm{d}x.
\]
Note that \(I_{n_1,m_2,m_3}\) and the right-hand side in the last identity are continuous linear forms on \(\redstatespace\).
The density of \(\redenergyspace\) in \(\redstatespace\) then allows to extend this identity to all of \(\redstatespace\), obtaining
\begin{equation}\label{equalityTMmodes}
  I_{n_1,m_2,m_3}(\vvu)=\nu (\lambda_\parallel+\lambda_\perp)\int_Q u_1(x) 
 \sin(k_2x_2) \sin(k_3x_3) \tilde w_{n_1} \, \,\mathrm{d}x,\qquad \forall \vvu\in \redstatespace.
\end{equation}
This identity also proves orthogonality of the TM modes in \(\redstatespace\) replacing \(\vvu\) by another TM mode. 
Its first component is a multiple of $\epsilon^{-1}\tilde w_{\ell_1}\sin(\ell_2\pi x_2)\sin(\ell_3\pi x_3)$, and $\epsilon^{-1}=\mu/\nu$.
So orthogonality follows from that of the Neumann eigenfunctions in $L^2_\mu(-1,1)$ and of the transverse sine functions.
Since TE modes have zero first component, \eqref{equalityTMmodes} also gives orthogonality between the two families.
The TE modes are mutually orthogonal by the corresponding transverse orthogonality and that of the Dirichlet eigenfunctions in $L^2_\epsilon(-1,1)$.
 
We now show that $\vvu\in \redstatespace$ vanishes if it is orthogonal to all listed eigenvectors.
 First the set  
 $\sin(k_2x_2) \sin(k_3x_3) \tilde w_n$ forms an orthogonal basis of $L^2_\mu(Q)$.
 The identity \eqref{equalityTMmodes} thus yields 
 $\mu^{-1} u_1=0$ or equivalently $u_1=0$.
 By means of the orthogonality of $\vvu$ to all TE-mode $\vvE^{\mathrm{TE}}_{n_1,m_2,m_3}$ for \(n_1\geq 1\) and \(m_2+ m_3>0\) and by   \eqref{TEmodesCD}, we directly obtain
\[
\int_{(-1,1)} \left(\int_{(0,1)^2} \vvu_\perp(x)\cdot \curl_\perp \left(\cos(k_2x_2) \cos(k_3x_3)\right) \, \,\mathrm{d}x_\perp\right)   w_{n_1}(x_1)
\epsilon \, \,\mathrm{d}x_1=0.
\]
Since  the set $\{w_n\}_{n\in \mathbb{N}}$ forms an orthonormal basis of $L^2_\epsilon(-1,1)$, we deduce that
\begin{equation}\label{eq:LastIdentity}
  \int_{(0,1)^2} \vvu_\perp(x)\cdot \curl_\perp \left(\cos(k_2x_2) \cos(k_3x_3)\right) \, \,\mathrm{d}x_\perp=0, \qquad
  \text{ a.e. } x_1\in (-1,1),
\end{equation}
for all \(m_2, m_3\in \mathbb{N}_0: m_2+ m_3>0\). 
Furthermore, as $\curl_\perp 1=0$, this identity remains valid  for all \(m_2, m_3\in \mathbb{N}_0$. 
Additionally, as the set
\(\{\cos(k_2x_2) \cos(k_3x_3)\}_{m_2, m_3\in \mathbb{N}_0}\) forms an orthogonal basis of \(H^1((0,1)^2)\), we deduce that
\begin{equation}\label{eq:LastIdentityonH1}
  \int_{(0,1)^2} \vvu_\perp(x)\cdot \curl_\perp v\, \,\mathrm{d}x_\perp=0, \qquad
   \forall v\in H^1((0,1)^2), \text{ a.e. } x_1\in (-1,1).
\end{equation}
Taking first \(\vvu_\perp\) in \(\mathcal{D}((0,1)^2)\), this equation implies that 
\begin{equation}\label{eq:curlfreeuperp}
\scurl_\perp \vvu_\perp(x_1,\cdot)=0 \text{ in } \mathcal{D}'((0,1)^2), \text{ a.e. } x_1\in (-1,1).
\end{equation}
Second, integrating by parts in \eqref{eq:LastIdentityonH1}, we also obtain
\begin{equation}\label{eq:tangentialcomponentuperp}
\vvu_\perp(x_1,\cdot)\cdot \tau=0 \text{ on } \partial (0,1)^2, \text{ a.e. } x_1\in (-1,1)
\end{equation}
for tangential vectors \(\tau\).
Observe that \(\div_{\perp} \vvu_{\perp} = 0\) as $u_1=0$ and $\div(\epsilon \vvu)=0$.
This last property combined with \eqref{eq:curlfreeuperp} and \eqref{eq:tangentialcomponentuperp} allows us to conclude that $\vvu_\perp=(0,0)$, because \((0,1)^2\) is simply connected.
\end{proof}

\subsection{Eigenmodes of the cavity}

It remains to compute the unobserved eigenpairs \eqref{eq:ObservedSplit} and prove Theorem~\ref{thm:FirstOrderStability}.
We continue to assume \eqref{partcasecoef} and use the TE and TM modes from \eqref{eq:TEModesgeneral} and \eqref{eq:TMModesgeneral}, which form an orthogonal basis by Theorem~\ref{thm:completeness}.

\begin{lem}\label{lem:UnobservedModes}
    Assume \eqref{partcasecoef}, and let \(\zeta\in(-1,1)\) with \(\zeta\notin\mathbb{Q}\).
    Then the unobserved eigenmodes can be chosen as
    \begin{equation*}
        \setc[\big]{\vvPhi_j}{j\in J_{\mathrm{u}}}
            = \setc[\big]{\vvE^{\mathrm{TM}}_{0,m_2,m_3}}{m_2,m_3\geq 1},
    \end{equation*}
    up to normalization. In particular, writing \(k_l=m_l\pi\) for \(l=2,3\), they have the form
    \begin{equation}\label{eq:UnobservedTMMode}
        \vvPhi_{m_2,m_3}
            = c_{m_2,m_3}\epsilon^{-1}(x_1)
                \sin(k_2x_2)\sin(k_3x_3)\,\vve_1,
        \qquad
        \beta_{0,m_2,m_3}^2 = \nu^{-1}(k_2^2+k_3^2),
    \end{equation}
    with nonzero real constants \(c_{m_2,m_3}\).
\end{lem}
\begin{proof}
    Fix \(n\geq 1\) and put \(q_n=\sqrt{\nu}\lambda_n>0\).
    The coefficients on both subintervals in \eqref{splitting_eigenvector} and \eqref{splitting_eigenvector_Neumann} are nonzero: otherwise the transmission conditions would force both the value and derivative on the other side of the interface to vanish, giving the zero eigenfunction.

    If \(w_n(\zeta)=0\), then \eqref{splitting_eigenvector} yields
    \begin{equation}\label{eq:SineZeroCondition}
        q_n(\zeta+1)=p\pi,
        \qquad
        q_n(1-\zeta)=r\pi
        \qquad\text{for some }p,r\in\N.
    \end{equation}
    Consequently, \(\zeta=(p-r)/(p+r)\in\mathbb{Q}\), a contradiction.
    Thus, \(w_n(\zeta)\neq 0\).
    Similarly, the transmission condition for \(\widetilde{w}_n\) yield
    \[
        d_n
            =\bigl((\epsilon^{(1)})^{-1}\partial_1\tilde w_n\bigr)(\zeta^-)
            =\bigl((\epsilon^{(2)})^{-1}\partial_1\tilde w_n\bigr)(\zeta^+).
    \]
    If \(d_n=0\), the derivatives from both sides vanish, and \eqref{splitting_eigenvector_Neumann} leads to \eqref{eq:SineZeroCondition} and the same contradiction.
    Hence \(d_n\neq 0\) for every \(n\geq 1\), whereas \(d_0=0\).

    The tangential traces of the TE and TM modes are, up to nonzero constants,
    \begin{align*}
        \bigl(\vvE^{\mathrm{TE}}_{n_1,m_2,m_3}\bigr)_\parallel
            &\simeq w_{n_1}(\zeta)
                \begin{pmatrix}
                    m_3\cos(k_2x_2)\sin(k_3x_3)\\
                    -m_2\sin(k_2x_2)\cos(k_3x_3)
                \end{pmatrix},\\
        \bigl(\vvE^{\mathrm{TM}}_{n_1,m_2,m_3}\bigr)_\parallel
            &\simeq d_{n_1}
                \begin{pmatrix}
                    k_2\cos(k_2x_2)\sin(k_3x_3)\\
                    k_3\sin(k_2x_2)\cos(k_3x_3)
                \end{pmatrix}.
    \end{align*}
    Therefore, every TE mode and every TM mode with \(n_1\geq 1\) has nonzero tangential trace, while the TM modes with \(n_1=0\) have zero tangential trace.

    We must also exclude cancellation of nonzero traces within an eigenspace.
    Traces with distinct transverse indices \((m_2,m_3)\) are orthogonal in \(L^2(F)^2\).
    For fixed transverse indices with \(m_2,m_3\geq 1\), the TE and TM traces are orthogonal as well. Indeed, suppressing normalization constants as above, we obtain
    \begin{equation*}
        \Bigl(\bigl(\vvE^{\mathrm{TE}}_{n_1,m_2,m_3}\bigr)_\parallel,
            \bigl(\vvE^{\mathrm{TM}}_{n_1,m_2,m_3}\bigr)_\parallel\Bigr)_{L^2(F)^2}
            \simeq \frac{w_{n_1}(\zeta)\overline{d_{n_1}}}{4}
                (m_3k_2-m_2k_3)=0.
    \end{equation*}
    Finally, for a fixed eigenvalue \(\beta^2\) and fixed transverse indices, \eqref{eq:Wavenumbers_general} determines \(\lambda_{n_1}^2\) uniquely. The positive one-dimensional eigenvalues are simple by Lemma~\ref{lem:BDir}.
    Thus, the nonzero traces within each eigenspace are pairwise orthogonal.

    If \(\{\vvPsi_r\}_{r=1}^N\) denotes the listed TE and TM modes with a fixed eigenvalue \(\beta^2\), then
    \begin{equation*}
        \left\|\left(\sum_{r=1}^N a_r\vvPsi_r\right)_\parallel\right\|_{L^2(F)^2}^2
            = \sum_{r=1}^N |a_r|^2
                \bigl\|(\vvPsi_r)_\parallel\bigr\|_{L^2(F)^2}^2,
        \qquad a_1,\ldots,a_N\in\C.
    \end{equation*}
    Hence, a linear combination has zero trace only if every coefficient multiplying a mode with nonzero trace vanishes.
    By Theorem~\ref{thm:completeness}, in each eigenspace the kernel of the trace is exactly the span of its TM modes with \(n_1=0\).
    Since \(\tilde w_0\) is constant, \eqref{eq:TMModesgeneral} gives \eqref{eq:UnobservedTMMode}.
\end{proof}

We can now evaluate the decay characterization \eqref{eq:FirstOrderDecaySet} on these modes.

\begin{proof}[Proof of Theorem~\ref{thm:FirstOrderStability}]
    By \eqref{eq:FirstOrderDecaySet}, the solution for  \(\vvu_0=(\vvE_0,\vvH_0)\) decays if and only if \(\mathcal{P}_{\mathrm{u}}\vvu_0=0\).
    Using \eqref{eq:FirstOrderUnobservedProjection} and \eqref{eq:LiftedEigenfunctions}, this is equivalent to
    \begin{equation*}
        \bigl(\vvu_0,\vvTheta_j^\pm\bigr)_{\statespace}
            = \beta_j\bigl(\vvE_0,\vvPhi_j\bigr)_{L_\epsilon^2(Q)^3}
                \mp i\bigl(\vvH_0,\curl\vvPhi_j\bigr)_{L^2(Q)^3}
            = 0,
        \qquad j\in J_{\mathrm{u}}.
    \end{equation*}
    Here the factor \(\mu\) in the magnetic part of the state-space inner product cancels the factor \(\mu^{-1}\) in the lifted eigenfunction.
    Adding and subtracting the equations for the two signs, and using \(\beta_j>0\), gives the equivalent conditions
    \begin{equation}\label{eq:FirstOrderOrthogonalityConditions}
        \bigl(\vvE_0,\vvPhi_j\bigr)_{L_\epsilon^2(Q)^3}=0,
        \qquad
        \bigl(\vvH_0,\curl\vvPhi_j\bigr)_{L^2(Q)^3}=0,
        \qquad j\in J_{\mathrm{u}}.
    \end{equation}

    By Lemma~\ref{lem:UnobservedModes}, we may use the real representatives
    \(\vvPhi_{m_2,m_3}=c_{m_2,m_3}\epsilon^{-1}g_{m_2,m_3}\vve_1\), where
    \[
        g_{m_2,m_3}(x_2,x_3)=\sin(k_2x_2)\sin(k_3x_3),
        \qquad m_2,m_3\geq 1.
    \]
    For the electric condition, Fubini's theorem yields
    \begin{equation*}
        \bigl(\vvE_0,\vvPhi_{m_2,m_3}\bigr)_{L_\epsilon^2(Q)^3}
            = c_{m_2,m_3}\int_{(0,1)^2}
                \Bigl(\int_{-1}^{1}\vvE_0\cdot\vve_1\,\mathrm{d}x_1\Bigr)
                g_{m_2,m_3}(x_2,x_3)\,\mathrm{d}(x_2,x_3).
    \end{equation*}
    Thus the weight \(\epsilon\) cancels the factor \(\epsilon^{-1}\) in the unobserved mode.
    The longitudinal electric moment belongs to \(L^2((0,1)^2)\) by the Cauchy--Schwarz inequality.
    Since the functions \(g_{m_2,m_3}\) form a complete orthogonal system in \(L^2((0,1)^2)\), the electric orthogonality conditions are equivalent to
    \[
        \int_{-1}^{1}\vvE_0\cdot\vve_1\,\mathrm{d}x_1=0
        \qquad\text{a.e.\ on }(0,1)^2.
    \]

    For the magnetic condition, we compute
    \begin{equation*}
        \curl\vvPhi_{m_2,m_3}
            = c_{m_2,m_3}\epsilon^{-1}(x_1)
                \bigl(0,\partial_3g_{m_2,m_3},-\partial_2g_{m_2,m_3}\bigr).
    \end{equation*}
    This identity holds in all of \(Q\), since the jump of \(\vvPhi_{m_2,m_3}\) is purely normal to \(F\) and hence produces no interface term in the curl.
    Using \(\epsilon^{-1}=\mu/\nu\) and the moments \(h_l=\int_{-1}^{1}\mu(\vvH_0\cdot\vve_l)\,\mathrm{d}x_1\) from the theorem, Fubini's theorem gives
    \begin{align*}
        \bigl(\vvH_0,\curl\vvPhi_{m_2,m_3}\bigr)_{L^2(Q)^3}
            &= \frac{c_{m_2,m_3}}{\nu}\int_{(0,1)^2}
                \bigl(h_2\partial_3g_{m_2,m_3}
                    -h_3\partial_2g_{m_2,m_3}\bigr)\,\mathrm{d}(x_2,x_3)\\
            &= \frac{c_{m_2,m_3}}{\nu}
                \bigl\langle\partial_2h_3-\partial_3h_2,g_{m_2,m_3}\bigr\rangle.
    \end{align*}
    Here \(h_2,h_3\in L^2((0,1)^2)\), so \(\partial_2h_3-\partial_3h_2\in H^{-1}((0,1)^2)\), and the bracket denotes its action on \(H^1_0((0,1)^2)\).
    The span of the Dirichlet eigenfunctions \(g_{m_2,m_3}\) is dense in \(H^1_0((0,1)^2)\).
    Hence all magnetic pairings vanish if and only if
    \(\partial_2h_3-\partial_3h_2=0\) in \(\mathcal{D}'((0,1)^2)\).
    Together with the electric condition, this proves the equivalence with \eqref{eq:FirstOrderMomentConditions}.
\end{proof}

\subsection*{Acknowledgment}
Julian Dörner and Roland Schnaubelt are funded by the Deutsche Forschungsgemeinschaft (DFG, German Research Foundation) – Project-ID 258734477 – SFB 1173.
JD thanks the French Institute for Research in Computer Science and Automation (Inria) in Lille for the kind hospitality during his research stay. This research stay was financially supported by the Karlsruhe House of Young Scientists (KHYS).

\bibliographystyle{amsalpha}
\bibliography{literature} %

\end{document}